\documentclass[runningheads]{llncs}
\usepackage[utf8]{inputenc}
\usepackage[T1]{fontenc}
\usepackage{lmodern}            

\usepackage{amsmath}
\usepackage{amssymb}
\usepackage{amsthm}
\usepackage{mathtools}

\usepackage{graphicx}
\usepackage{xcolor}

\usepackage{booktabs}

\usepackage[ruled,linesnumbered]{algorithm2e}

\usepackage{float}

\usepackage{microtype}

\usepackage{url}
\usepackage{hyperref}

\usepackage{amsmath,amsfonts,bm}

\def\eqref#1{equation~\ref{#1}}

\def\1{\bm{1}}

\DeclareMathAlphabet{\mathsfit}{\encodingdefault}{\sfdefault}{m}{sl}
\SetMathAlphabet{\mathsfit}{bold}{\encodingdefault}{\sfdefault}{bx}{n}

\newtheoremstyle{dotstyle}
{\topsep}      
{\topsep}      
{\itshape}     
{}             
{\bfseries}    
{.}            
{0.5em}        
{}             

\newtheoremstyle{dotstyle_def}
{\topsep}
{\topsep}
{\normalfont}
{}
{\bfseries}
{.}
{0.5em}
{}

\theoremstyle{plain}

\theoremstyle{definition}

\theoremstyle{dotstyle}
\newtheorem{assumption}{Assumption}

\newcommand{\norm}[1]{\left\lVert #1 \right\rVert}
\newcommand{\inprod}[2]{\left\langle #1, #2 \right\rangle}

\providecommand{\Proj}{\mathrm{Proj}}

\SetKwInput{Input}{Input}
\SetKwInput{Output}{Output}

\renewcommand{\eqref}[1]{\hyperref[#1]{(\ref*{#1})}}

\begin{document}
\title{First-Order Optimization Methods with Generalized Smoothness for Low-Rank Matrix Factorization}
%
%

\author{A.\;V.\;Bidanets \inst{1}\orcidID{0009-0002-1895-1413} \and
 F.\;S.\;Stonyakin \inst{1}\orcidID{0000-0002-9250-4438}}
\authorrunning{A.V. Bidanets, F. S. Stonyakin}
%
\institute{
Moscow Institute of Physics and Technology, Russia}
\maketitle              
\begin{abstract}
	This paper studies several first-order methods for low-rank matrix minimization in the Burer--Monteiro form, where the target matrix is represented as the product of a nonnegative factor matrix and its transpose.
	Such a factorization destroys convexity of the objective function and global Lipschitz continuity of the gradient. 
	We prove that for a quadratic objective function that is strongly convex only on the
	subspace of symmetric matrices, the Burer--Monteiro factorized objective
	possesses generalized smoothness with a fractional exponent 2/3.
	On this basis, three methods are proposed: an adaptive gradient descent method with Armijo backtracking; an unconstrained SafeAdam method combining the Adam direction with a descent-angle certificate and a fallback to the normalized gradient; and a certified projected version for the nonnegativity constraint. 
	Global convergence and explicit convergence rates for the proposed methods have been established.
	The obtained convergence rate is better than the convergence rate of ordinary unaccelerated gradient descent in the case of a sufficiently large gradient.
	In the small-gradient regime, the convergence guarantee is of square-root type, and the corresponding complexity scales as the inverse square of the desired accuracy.
	In the large-gradient regime, the convergence guarantee is of three-quarters type, and the corresponding complexity is proportional to the inverse of the accuracy raised to the power of 4/3.
	Experiments on SymNMF using the TDT2, Reuters and CBCL datasets, which are standard benchmarks for document and image clustering, show that the methods match or outperform Dyn-NoLips without tuning the relative-smoothness kernel parameters.
\end{abstract}

\newpage
\tableofcontents
\newpage

\section*{Acknowledgments}
The research is supported by the Ministry of Science and Higher Education of the Russian Federation, project No. FSMG-2024-0011

\section{Preliminaries and Problem Statement}

\subsection{Notation}
For a square matrix $M$ we denote its trace by $\mathrm{Tr}\, M = \sum_{i=1}^n M_{ii}$.
For two matrices $X$ and $Y$ of the same dimension we denote the standard Euclidean inner product and norm by $\langle X, Y \rangle = \mathrm{Tr}(X^{\top} Y)$ and $\|X\| = \sqrt{\mathrm{Tr}(X^{\top} X)}$ --- the Frobenius norm.

For a linear operator $\mathcal{T}: \mathbb{R}^{n\times r}\to\mathbb{R}^{n\times r}$ 
(in particular, for the Hessian $\nabla^2 f(X)$) we use the same notation 
$\|\mathcal{T}\|$ for the operator norm induced by the Frobenius norm:
$\|\mathcal{T}\| = \sup_{\|U\|=1} \|\mathcal{T}U\|$.

For a function $f: \mathbb{R}^{n \times r} \to \mathbb{R}$ we denote by $\nabla f(X)$ its gradient matrix $(\nabla f(X))_{ij} = \frac{\partial f(X)}{\partial x_{ij}}$ and by $\nabla^2 f(X)[U, V]$ the second derivative at the point $X$ along the directions $U, V \in \mathbb{R}^{n \times r}$.

$I_r$ denotes the identity matrix of dimension $r \times r$.

For two matrices $X, Y$ we say that $X \preceq Y$ if the matrix $Y - X$ is positive semidefinite.

For a matrix $X$ we denote by $\mathrm{Proj}_{\ge 0}(X)$ the entrywise projection onto the nonnegative orthant as follows:
\[
(\mathrm{Proj}_{\ge 0}(X))_{ij} = \max\{X_{ij},\,0\},\qquad i=1,\dots,n,\;j=1,\dots,r.
\]

For a matrix $Z=(Z_{ij})$ we define the entrywise infinity norm by
\[
\|Z\|_{\infty} := \max_{i,j}|Z_{ij}|.
\]

\subsection{Assumptions}

\begin{assumption}[Strong convexity of $F$ on symmetric matrices]
	\label{asm:quadratic}
	The function \(F \colon \mathbb{R}^{n \times n} \to \mathbb{R}\) is quadratic:
	\begin{equation}\label{eq_quadratic_f}
		F(Y) = \frac{1}{2} \langle \mathcal{A} Y, Y \rangle + \langle B, Y \rangle,
	\end{equation}
	where
	\begin{itemize}
		\item \(\mathcal{A} \colon \mathbb{R}^{n \times n} \to \mathbb{R}^{n \times n}\) is a linear self-adjoint operator with respect to the Frobenius inner product
		\[
		\langle X, Y \rangle = \mathrm{Tr}(X^\top Y);
		\]
		\item \(\mathcal{A}\) maps symmetric matrices into symmetric ones: 
		\(\mathcal{A}(Y) \in \mathbb{S}^{n \times n}\) for all \(Y \in \mathbb{S}^{n \times n}\);
		\item \(B \in \mathbb{S}^{n \times n}\) is a symmetric matrix.
	\end{itemize}
	It is assumed that the operator \(\mathcal{A}\) is \(\mu\)-positive definite \textbf{only on the subspace of symmetric matrices} \(\mathbb{S}^{n \times n}\), that is,
	\[
	\langle \mathcal{A} Y, Y \rangle \ge \mu \|Y\|_F^2 
	\qquad \forall \, Y \in \mathbb{S}^{n \times n}.
	\]
	This is equivalent to the \(\mu\)-strong convexity of the function \(F\) on \(\mathbb{S}^{n \times n}\). 
	The positive definiteness does not extend to the whole space \(\mathbb{R}^{n \times n}\), which constitutes the key relaxation compared with the work \cite{Dragomir21}.
\end{assumption}

\begin{remark}
	The symmetry of the matrix $B$ and the property of the operator $\mathcal{A}$ to preserve symmetry are natural for a wide class of low-rank optimization problems, where the target matrix $Y$ is symmetric by its nature:
	symmetric matrix completion,
	the phase retrieval problem,
	estimation of covariance matrices,
	Euclidean distance matrix completion,
	nuclear norm and spectral regularization.
\end{remark}

\begin{assumption}[Regularizer and well-posedness]
	\label{asm:regularizer}
	The regularizer \(R \colon \mathbb{R}^{n \times r} \to \mathbb{R} \cup \{+\infty\}\) is closed, proper, and convex.
	Moreover, the objective function \(\Psi\) is bounded below on the feasible set:
	\[
	\inf_{X \in \mathbb{R}^{n \times r}_{\ge 0}} \Psi(X) > -\infty.
	\]
\end{assumption}

\begin{remark}[Gradient properties of \(F\) on symmetric matrices]
	\label{rem:F-gradient-properties}
	Under Assumption~\ref{asm:quadratic}, for all \(X,Y\in\mathbb{S}^{n\times n}\),
	\begin{equation}
		\langle \nabla F(X)-\nabla F(Y), X-Y\rangle
		\ge \mu\|X-Y\|_F^2,
	\end{equation}
	and
	\begin{equation}
		\|\nabla F(X)-\nabla F(Y)\|_F
		\le L_F\|X-Y\|_F.
	\end{equation}
	Here \(L_F\) denotes the Lipschitz constant of \(\nabla F\) on the subspace of symmetric matrices \(\mathbb{S}^{n\times n}\). 
	For the quadratic function \eqref{eq_quadratic_f}, one may take
	\[
	L_F = \|\mathcal{A}\|_{\mathbb{S},\mathrm{op}},
	\]
	where the restricted operator norm \(\|\mathcal{A}\|_{\mathbb{S},\mathrm{op}}\) is discussed in Remark~\ref{rem:norm-A}.
\end{remark}

\begin{remark}[On the constant $L_F$ and the symmetric operator norm]
	\label{rem:norm-A}
	For the quadratic function \eqref{eq_quadratic_f}
	the gradient equals
	\[
	\nabla F(Y)=\mathcal{A}Y+B
	\]
	for $Y\in\mathbb{S}^{n\times n}$. Hence, for all $Y,Z\in\mathbb{S}^{n\times n}$,
	\[
	\|\nabla F(Y)-\nabla F(Z)\|_F
	=
	\|\mathcal{A}(Y-Z)\|_F
	\le
	\|\mathcal{A}\|_{\mathbb{S},\mathrm{op}}\|Y-Z\|_F,
	\]
	where we introduce the operator norm of the restriction of $\mathcal{A}$ to the subspace of symmetric matrices:
	\[
	\|\mathcal{A}\|_{\mathbb{S},\mathrm{op}}
	:=
	\sup_{\substack{Y\in\mathbb{S}^{n\times n}\\ \|Y\|_F=1}}
	\|\mathcal{A}Y\|_F.
	\]
	This is the natural norm for the factorized function $f(X)=F(XX^\top)$, because in the Hessian of $f$ the operator $\mathcal{A}$ is applied only to symmetric matrices, namely to $XX^\top$ and to $XU^\top+UX^\top$. Therefore all subsequent estimates can be written with the restricted norm $\|\mathcal{A}\|_{\mathbb{S},\mathrm{op}}$.
	
	In this sentence, \(U\) is the direction matrix along which the Hessian of the function \(f\) at the point \(X\) acts. It has the same size as \(X\):
	\[
	U \in \mathbb{R}^{n \times r}.
	\]
	When the second derivative of \(f(X)=F(XX^\top)\) along the direction \(U\) is computed, the symmetric matrix
	\[
	XU^\top + UX^\top \in \mathbb{S}^{n \times n},
	\]
	appears in the expression, and the operator \(\mathcal{A}\) is applied to it. This is precisely why the norm of \(\mathcal{A}\) is taken on the subspace of symmetric matrices \(\mathbb{S}^{n \times n}\), and not on the whole space \(\mathbb{R}^{n \times n}\).
	
	That is:
	\begin{itemize}
		\item \(X \in \mathbb{R}^{n \times r}\) is the point at which the Hessian is computed;
		\item \(U \in \mathbb{R}^{n \times r}\) is the direction along which the Hessian acts;
		\item \(XU^\top + UX^\top\) is the symmetric matrix arising from differentiating \(XX^\top\).
	\end{itemize}

	The full operator norm on $\mathbb{R}^{n\times n}$ may also be used, since
	\[
	\|\mathcal{A}\|_{\mathbb{S},\mathrm{op}}
	\le
	\sup_{\substack{W\in\mathbb{R}^{n\times n}\\ \|W\|_F=1}}
	\|\mathcal{A} W\|_F,
	\]
	where \(W \in \mathbb{R}^{n \times n}\) is an arbitrary test matrix of size \(n \times n\) used to define the full operator norm.
	
	But it generally gives more conservative constants and may depend on the behavior of $\mathcal{A}$ on the skew-symmetric subspace, which does not affect the function $f(X)=F(XX^\top)$.
\end{remark}

\subsection{Problem Statement}
Following the works~\cite{Dragomir21,Burer05,Candes09,Chen15,Fang12,Mishra11,Recht07,Jain13,Cai10}, consider the low-rank matrix minimization problem. We first state the original matrix problem in a general form and then describe the factorized approximation actually used in the paper.

The general symmetric-matrix problem is
\begin{equation}\label{original_problem}
	\min_{Y \in \mathbb{S}^{n \times n}} F(Y),
\end{equation}
where
$F \colon \mathbb{R}^{n \times n} \to \mathbb{R}$
is a smooth function, convex on the subspace of symmetric matrices
$\mathbb{S}^{n \times n}$.
For a correct definition of the derivatives of the factorized function $f(X)=F(XX^\top)$ it is convenient to assume that $F$ is defined on the whole space $\mathbb{R}^{n \times n}$.

If the target application requires a low-rank positive semidefinite and entrywise nonnegative matrix, the corresponding structured problem is
\begin{equation}\label{original_problem_constrained}
	\min_{Y \in \mathbb{S}^{n \times n}} F(Y)
	\quad\text{subject to}\quad
	Y \succeq 0, \; Y \ge 0, \; \operatorname{rank}(Y) \le r.
\end{equation}
Here $Y \ge 0$ denotes entrywise nonnegativity.

In the present paper we do not solve
\eqref{original_problem} or \eqref{original_problem_constrained}
directly. Instead, following the Burer--Monteiro approach \cite{Dragomir21}, we solve the factorized approximation
\begin{equation}\label{main_func}
	\min_{X \in \mathbb{R}^{n \times r}_{\ge 0}} \Psi \left(X\right),
	\qquad
	\Psi \left(X\right) := F \left(X X^{\top}\right) + R \left(X\right),
\end{equation}
where
$X \in \mathbb{R}^{n \times r}_{\ge 0}$,
$f \left( X \right) := F \left(X X^{\top}\right)$,
$f: \mathbb{R}^{n \times r} \rightarrow \mathbb{R}$;
$r \le n$;
$R \left( X \right)$ is a \emph{simple} convex regularizer, for example:
\begin{itemize}
	\item the $\ell_1$-norm $R\left(X\right) = {\left \lVert X \right \rVert}_{1}$,
	\item the quadratic $\ell_2$-regularization $R\left(X\right) = \frac{\lambda}{2} {\left \lVert X \right \rVert}^{2}$.
\end{itemize}

The change of variables $Y = XX^{\top}$ with $X \ge 0$ restricts the feasible set to
\begin{equation}\label{eq:feasible_set_Y}
	\mathcal{Y}_r^+
	:=
	\left\{
	XX^{\top}
	\; : \;
	X \in \mathbb{R}^{n \times r}_{\ge 0}
	\right\}.
\end{equation}
Every matrix $Y \in \mathcal{Y}_r^+$ satisfies
\begin{equation}\label{eq:Y_properties}
	Y \succeq 0,
	\qquad
	Y \ge 0,
	\qquad
	\operatorname{rank}(Y) \le r.
\end{equation}
Consequently,
\[
\mathcal{Y}_r^+
\subseteq
\left\{
Y \in \mathbb{S}^{n \times n}
\; : \;
Y \succeq 0, \; Y \ge 0, \; \operatorname{rank}(Y) \le r
\right\}.
\]
Thus, unless the original problem is already restricted to the set $\mathcal{Y}_r^+$, problem \eqref{main_func} should be understood as a low-rank factorized approximation/restriction of the original problem, not as an equivalent reformulation of unconstrained minimization over $\mathbb{S}^{n \times n}$.

\begin{remark} In the theoretical part of the paper we consider only the case 
	$R \equiv 0$. 
	Then the projection step $\mathrm{Proj}_{\ge 0}$ is correct. 
	For general $R$ (for example, $\ell_1$ or $\ell_2$) one should replace the projection by 
	the proximal operator; the analysis of such regularizers is beyond the scope of this work.
\end{remark}

\subsection{Motivation}
\paragraph{Benefits of low-rank factorization}
The Burer--Monteiro low-rank factorization, in which a symmetric matrix \(Y\in\mathbb{S}^{n\times n}\) is represented as \(Y=XX^\top\) with \(X\in\mathbb{R}^{n\times r}\) and \(r\ll n\), offers several substantial advantages. First, it reduces the number of optimization variables from \(O(n^2)\) to \(O(nr)\), which is essential for large-scale problems. Second, when the factor \(X\) is constrained to be entrywise nonnegative, the factorization automatically guarantees that \(Y\) is positive semidefinite, entrywise nonnegative, and has rank at most \(r\). This structural enforcement eliminates the need for expensive projections onto the cone of positive semidefinite matrices or explicit rank constraints, which are computationally prohibitive in high dimensions. Third, although the resulting factorized problem \(\min_X f(X)=F(XX^\top)\) is nonconvex, it often possesses favorable properties such as generalized smoothness, enabling the design and analysis of first-order methods with rigorous convergence guarantees. Consequently, low-rank factorization provides a computationally efficient and theoretically tractable framework for problems such as symmetric nonnegative matrix factorization, where the target matrix is naturally symmetric and nonnegative.
\paragraph{Limitations of the bregman approach and adaptability of the proposed method}\label{limitation_Dyn_NoLips}
In the considered experiments \cite{Dragomir21} on the SymNMF problem, the Dyn-NoLips method demonstrates a high convergence rate. This is largely due to the specific structure of the quadratic loss function $F(Y)=\frac12\|M-Y\|_F^2$, for which the linear operator $\mathcal{A}$ is the identity map. Consequently, the global Lipschitz constant $L_F=1$ and the strong convexity parameter $\mu_F=1$ are known exactly, allowing the theoretically justified kernel parameters $\alpha=6L_F$ and $\sigma=2\|\nabla F(0)\|$ to be set without any heuristic search.
In more general low-rank optimization problems, the global spectral constants $L_F$ and $\mu_F$ are often unknown a priori or vary significantly across the domain. This reveals a fundamental limitation of the Bregman gradient approach:
\paragraph{Rigidity of the bregman kernel geometry.}
In the NoLips framework \cite{Dragomir21}, the parameters $\alpha$ and $\sigma$ define the distance kernel $h_N(X) = \frac{\alpha}{4}\|X\|^4 + \frac{\sigma}{2}\|X\|^2$. This kernel generates the Bregman divergence and the underlying non-Euclidean geometry of the space. The relative smoothness condition $f(X) \le f(Y) + \langle \nabla f(Y), X-Y \rangle + L D_h(X,Y)$ must hold globally. Therefore, $\alpha$ and $\sigma$ must be chosen based on global worst-case bounds for $L_F$ and $\|\nabla F(0)\|$. Adaptive adjustment of these parameters would mean changing the metric space at every iteration. This breaks the standard convergence theory of Bregman methods. Consequently, Dyn-NoLips is forced to use conservative global constants. As a result, its step sizes are suboptimal.
In practice, when these global constants are unknown, the parameters
$\alpha$ and $\sigma$ of the Bregman kernel must be selected by an
external tuning procedure, typically a grid search over a finite set
of candidate values, since the convergence theory does not allow them
to be adapted during the iterations.
\paragraph{Adaptability of the generalized smoothness framework.}
In contrast, our approach is based on the generalized $(2/3, L_0, L_{2/3})$-smoothness property in the standard Euclidean geometry. The constants $L_0$ and $L_{2/3}$ are used \emph{exclusively for computing the step size} via the integral $q$-function; they do not define the geometry of the space itself. This decoupling provides a crucial advantage: the constants can be \textbf{estimated and updated adaptively on the fly} (e.g., using local Hessian approximations, current gradient norms, or adaptive backtracking as in Algorithm~\ref{alg:adaptive_gd_backtrack}) without violating the convergence guarantees. 

Thus, while both frameworks theoretically rely on the spectral properties of the problem, the Bregman approach requires global a priori knowledge to fix the geometry, whereas our method can operate with purely local, adaptive estimates of the smoothness constants. Combined with the computational efficiency of standard Euclidean steps (avoiding the inner subproblems of the Gram kernel) and the use of adaptive search directions (SafeAdam-$\ell$), our approach is highly robust on ill-conditioned or ravine-like landscapes where fixed-kernel Bregman methods struggle.

\subsection{Contribution of this paper}
\label{sec:contributions}

The central theoretical contribution of this work is the proof that the
Burer--Monteiro factorized objective $f(X)=F(XX^\top)$, where $F$ is a
quadratic function, satisfies the generalized $(2/3,\,L_0,\,L_{2/3})$-smoothness
condition with explicit constants expressed through the restricted operator
norm $\|\mathcal{A}\|_{\mathbb{S},\mathrm{op}}$, the strong convexity
parameter $\mu$, the rank $r$, and $\|B\|_F$. A key relaxation compared with
the prior work~\cite{Dragomir21} is that the $\mu$-positive definiteness of
$\mathcal{A}$ is required \emph{only on the subspace of symmetric matrices}
$\mathbb{S}^{n\times n}$, not on the entire space $\mathbb{R}^{n\times n}$,
which is natural for problems where the target matrix is symmetric by
construction. Furthermore, the general $\ell$-smooth framework of
Tyurin~\cite{tyurin2025toward} assumes that the smoothness modulus $\ell$ is
locally Lipschitz; in our setting $\ell(g)=L_0+L_{2/3}g^{2/3}$ this condition
fails at $g=0$ whenever $L_0=0$. We show that all the key analytical
tools---the gradient-difference estimate, the integral function-increment
bound, and the convergence analysis---remain valid under the substantially
weaker assumptions of continuity, nonnegativity, monotonicity, and integrability
of $\ell$, thereby extending the applicability of the generalized smoothness
theory to a wider class of problems.

Three first-order algorithms are proposed for the factorized problem. The
first is an adaptive gradient descent method with Armijo backtracking
(Algorithm~\ref{alg:adaptive_gd_backtrack}), which starts from the
theoretically justified step $\eta_t^{\mathrm{theory}}=q(a_t;a_t)/a_t$ and
adaptively enlarges it by a factor $\alpha_t\in[1,5]$ while maintaining a
sufficient-descent guarantee. The second is an unconstrained SafeAdam-$\ell$
method (Algorithm~\ref{alg:safeadam_unconstrained}), which combines the Adam
search direction with a descent-angle certificate and a fallback to the
normalized gradient step, with the step size determined by the $q$-function
of generalized smoothness. The third is a certified projected SafeAdam-$\ell$
(Algorithm~\ref{alg:safeadam_projected}) for the entrywise nonnegativity
constraint $X\ge 0$, whose convergence is established in the gradient-mapping
norm $\|G_t\|_F$ with an explicit $O(T^{-1/2})$ rate.
The analysis of the projected method relies on the boundedness of the
initial level set and of the gradients along the generated trajectory
(Lemma~\ref{lem:bounded_gradients}), and the convergence in the
gradient-mapping norm is shown to be locally equivalent to convergence
in the projected gradient norm.

For both the adaptive GD with backtracking and the unconstrained SafeAdam-$\ell$
method we establish a two-regime convergence analysis. By splitting the
iterations into small-gradient and large-gradient sets via a two-stage
telescoping argument, we obtain the rates
$\min_t\|\nabla f(X^t)\|_F = O(T^{-1/2})$ in the small-gradient regime and
$\min_t\|\nabla f(X^t)\|_F = O(T^{-3/4})$ in the large-gradient regime,
corresponding to complexities $O(\varepsilon^{-2})$ and
$O(\varepsilon^{-4/3})$, respectively. The large-gradient rate strictly
improves upon the standard $O(T^{-1/2})$ bound of unaccelerated gradient
descent under global $L$-smoothness. The additive no-max complexity bounds
are unconditional and do not require monotonicity of the gradient norm,
quasiconvexity, or any gradient-dominance condition.

Unlike the Bregman-based Dyn-NoLips method~\cite{Dragomir21}, which requires
global a priori knowledge of the spectral constants $L_F$ and $\mu_F$ to fix
the geometry of the distance kernel (and hence necessitates a grid search
over the kernel parameters in general problems), our Euclidean
generalized-smoothness framework uses the constants $L_0$ and $L_{2/3}$
exclusively for step-size computation. This decoupling allows on-the-fly
adaptive estimation of the smoothness constants without violating convergence
guarantees, making the proposed methods robust on ill-conditioned or
ravine-like landscapes where fixed-kernel Bregman methods struggle.

Numerical experiments on the SymNMF problem with the TDT2, Reuters, and CBCL
datasets confirm the theoretical predictions. On text data with a predictable
landscape, adaptive backtracking reduces the iteration count by a factor of
approximately five compared with the basic adaptive GD and matches or
outperforms Dyn-NoLips. On face images with a ravine-like landscape,
SafeAdam-$\ell$ achieves the best convergence rate and the minimal wall-clock
time, substantially outperforming both GD variants and Dyn-NoLips. In all
cases the proposed methods do not require tuning of the relative-smoothness
kernel parameters.

\section*{Structure of the Paper}
The paper consists of an introduction and nine main sections, in addition to an appendix. In Sect.~1, we introduce the matrix notation, the restricted operator norm on symmetric matrices, and the problem statement, motivating the Euclidean generalized smoothness framework. Sect.~2 is devoted to the property of generalized smoothness, where we prove the $(2/3, L_0, L_{2/3})$-smoothness inequality with explicit constants. In Sect.~3, we define the $q$-transform of the smoothness modulus and establish the integral estimates for the function increment. Sect.~4 presents the adaptive gradient descent method with Armijo backtracking, along with its convergence analysis and two-regime complexity bounds. In Sect.~5, we propose the unconstrained SafeAdam-$\ell$ method, combining the Adam direction with a descent-angle certificate, and prove its global convergence. Sect.~6 introduces the projected certified Adam method and analyzes its convergence in the gradient-mapping norm. In Sect.~7, we discuss the relation to the gradient projection method and explain the local equivalence between the gradient mapping and the gradient projection. Sect.~8 presents numerical experiments on symmetric nonnegative matrix factorization, comparing the proposed methods with the Dyn-NoLips baseline on text and image datasets. In the last Section~9, we review the obtained results and summarize the two-regime convergence rates. In the appendix, we provide the pseudocode for the SafeAdam-$\ell$ algorithms and list the problem classes where the grid search for Bregman kernel parameters is unavoidable.

\section{The Property of Generalized Smoothness of the Objective Function}
\label{sec:generalized_smoothness}

\begin{definition}[\cite{liu2025warmup}]
	A function \(f\) is called \emph{\((p, L_0, L_p)\)-smooth} if there exist \(L_0 \ge 0\), \(L_p \ge 0\) and \(p\in[0,1)\) such that
	\begin{equation}\label{def_gensmooth}
		\|\nabla^2 f(X)\|_{\mathrm{op}} \le L_0 + L_p \|\nabla f(X)\|_F^p, \quad X\in\mathbb{R}^{n\times r},
	\end{equation}
	where $\|\nabla^2 f(X)\|_{\mathrm{op}}$ denotes the operator norm of the Hessian induced by the Frobenius norm.
	In the case \(p<1\) the convergence of all the integrals used is preserved also for \(L_0=0\), 
	since the function \(g^{-p}\) is integrable at zero.
\end{definition}

\begin{remark}
	For \(p=0\) this is the classical \(L\)-smoothness; the case \(p<1\) reflects a sublinear growth of the Hessian with respect to the growth of the gradient.
\end{remark}

\begin{remark}[Two regimes of the smoothness model]
	The constant \(L_0\) and the term \(L_p\|\nabla f(X)\|^p\) play different
	roles depending on the magnitude of the gradient. When \(\|\nabla f(X)\|\)
	is small, the constant part \(L_0\) dominates, and the Hessian bound
	behaves as in the classical \(L\)-smoothness case. When \(\|\nabla f(X)\|\)
	is large, the fractional-power term \(L_p\|\nabla f(X)\|^p\) dominates, so
	the Hessian grows sublinearly with the gradient. As shown below, this
	dichotomy is inherited by the convergence rates: the small-gradient regime
	yields a square-root-type guarantee, whereas the large-gradient regime
	yields a three-quarters-exponent guarantee.
\end{remark}

\begin{remark}[Relaxed requirements on $\ell$]
	In the general theory of $\ell$-smoothness \cite{tyurin2025toward} it is assumed that the function 
	$\ell:[0,\infty)\to(0,\infty)$ is nondecreasing, positive and 
	{\it locally Lipschitz}. 
	For the case considered by us 
	\[
	\ell( g ) = L_0 + L_{2/3} g^{2/3},
	\qquad
	L_0 \ge 0,\quad L_{2/3} \ge0,
	\qquad
	L_0 + L_{2/3} > 0,
	\]
	where $g$ is the norm of the gradient at the point under consideration,
	local Lipschitz continuity fails at zero if \(L_0=0\). 
	However, in the present work we show that in order to obtain all the key results it suffices to have the following relaxed conditions:
	\begin{itemize}
		\item \(\ell\) is continuous and nonnegative on \([0,\infty)\);
		\item \(\ell\) is positive on \((0,\infty)\);
		\item \(\ell\) is nondecreasing;
		\item for any \( a \ge 0\) the integral 
		\[
		\int_{0}^{s}\frac{dv}{\ell(a+v)}
		\]
		converges for all \(s\ge0\).
	\end{itemize}
	For \(\ell(g) = L_0+L_{2/3} g^{2/3}\) the convergence of the integral is guaranteed by the fact 
	that either \(L_0>0\), or \(L_0=0\), and then the singularity at zero has the form
	\[
	\frac{1}{L_{2/3}v^{2/3}},
	\]
	and the integral of \(v^{-2/3}\) is finite. Therefore all the references to the results 
	for \(\ell\)-smooth functions from \cite{tyurin2025toward} remain legitimate.
\end{remark}

We show that for $f$ the following estimate of the Hessian holds:
\begin{equation}
	\left \lVert \nabla^2 f\left(X\right) \right \rVert \le L_0 + L_{2/3} {\left \lVert \nabla f\left(X\right) \right \rVert}^{2/3},
\end{equation}

\begin{remark}
	The original function $F$ is $\mu$-strongly convex (see Assumption~\ref{asm:quadratic}),
	however the factorized function $f(X)=F(XX^\top)$ is \textbf{not} convex (and all the more not
	strongly convex). 
	The strong convexity of $F$ is used (without loss of practical applicability) exclusively as a technical
	tool for obtaining a lower bound on the norm of $\nabla f(X)$ (Theorem~\ref{thm:lower-bound}),
	which then allows one to prove the generalized smoothness property of $f$.
\end{remark}

The function $f(X) = F(XX^\top)$ for a quadratic $F$ is a polynomial of the fourth degree 
(not taking into account the regularizer $R$). 

Let \(F(Y) = \frac12 \langle \mathcal{A}Y,Y\rangle + \langle B,Y\rangle\), where \(\mathcal{A}\) satisfies Assumption~\ref{asm:quadratic}, i.e., \(\mathcal{A}\) is \(\mu\)-positive definite on the subspace of symmetric matrices \(\mathbb{S}^{n\times n}\). 
Then \(f(X) = \frac12 \langle \mathcal{A}(XX^\top), XX^\top\rangle + \langle B, XX^\top\rangle\) 
is not homogeneous in the general case (except for the case $B=0$), since it contains 
simultaneously terms of degree 4 and of degree 2 in $X$.

Direct computations (cf.~\cite{Dragomir21}) give, for the quadratic function \eqref{eq_quadratic_f},
\[
\|\nabla^2 f(X)\|
\le 6\|\mathcal{A}\|_{\mathbb{S},\mathrm{op}}\|X\|_F^2 + 2\|\nabla F(0)\|_F
= 6\|\mathcal{A}\|_{\mathbb{S},\mathrm{op}}\|X\|_F^2 + 2\|B\|_F.
\]
Here \(\|\mathcal{A}\|_{\mathbb{S},\mathrm{op}}\) is the operator norm of the restriction of \(\mathcal{A}\) to the subspace of symmetric matrices, introduced in Remark~\ref{rem:norm-A}. This restricted norm is natural because the Hessian of the factorized function applies \(\mathcal{A}\) only to symmetric matrices. If desired, one may replace it by the full operator norm on \(\mathbb{R}^{n\times n}\), obtaining a valid but generally more conservative bound.

If \(L_F\) denotes a Lipschitz constant of \(\nabla F\) on the subspace of symmetric matrices \(\mathbb{S}^{n \times n}\), then for the quadratic function \(F\) one may take
\[
L_F = \|\mathcal{A}\|_{\mathbb{S},\mathrm{op}}.
\]

\begin{remark}[Twice differentiability of $f$] 
	Since $F$ is a quadratic function, it belongs to the class $C^\infty$; the map $\phi(X)=XX^\top$ is polynomial, hence it also belongs to the class $C^\infty$. 
	Consequently, $f = F\circ\phi$ belongs to the class $C^\infty$, and in particular is twice continuously differentiable.
\end{remark}

\begin{lemma}\label{lem:norm_XX}
	For any matrix \(X\in\mathbb{R}^{n\times r}\) it holds that
	\[
	\|XX^\top\|_F \le \|X\|_F^2.
	\]
\end{lemma}
\begin{proof}
	Let \(X = U \Sigma V^\top\) be a full singular value decomposition,
	where \(U\in\mathbb{R}^{n\times n}\) and \(V\in\mathbb{R}^{r\times r}\) are orthogonal,
	and \(\Sigma\in\mathbb{R}^{n\times r}\) is rectangular diagonal with nonnegative
	diagonal entries \(\sigma_1,\dots,\sigma_p\), where \(p = r \).
	Then \(XX^\top = U\Sigma\Sigma^\top U^\top\) and
	\[
	\|XX^\top\|_F^2 = \mathrm{Tr} \left(\Sigma\Sigma^\top\Sigma\Sigma^\top\right)
	= \sum_{i=1}^p \sigma_i^4
	\le \left( \sum_{i=1}^p \sigma_i^2 \right)^2
	= \|X\|_F^4,
	\]
	since all \(\sigma_i^2\ge0\). Taking the square root, we obtain the required statement.
\end{proof}

\begin{theorem}[Lower bound on the gradient norm for $f(X)=F \left(XX^\top \right)$]
	\label{thm:lower-bound}
	Let $F(Y)=\frac12\langle\mathcal{A}Y,Y\rangle + \langle B,Y\rangle$, where
	\begin{itemize}
		\item $\mathcal{A}:\mathbb{R}^{n\times n}\to\mathbb{R}^{n\times n}$ is a linear self-adjoint operator mapping $\mathbb{S}^{n\times n}$ into itself;
		\item $B\in\mathbb{S}^{n\times n}$ is a symmetric matrix.
	\end{itemize}
	We assume that $F$ is \emph{strongly convex} with parameter $\mu>0$, i.\,e.
	\begin{equation}\label{strong-convex}
		\langle\mathcal{A}Y,Y\rangle \ge \mu\|Y\|_F^2 \qquad \forall\,Y\in\mathbb{S}^{n\times n}.
	\end{equation}
	Then for any \(X\in\mathbb{R}^{n\times r}\) it holds that
	\begin{equation}\label{final-lower}
		\|\nabla f(X)\|_F \ge \frac{2\mu}{r}\|X\|_F^3 - 2\|B\|_F\|X\|_F.
	\end{equation}
	In particular, if \(B=0\), the inequality holds for all \(X\) and takes the form
	\[
	\|\nabla f(X)\|_F \ge \frac{2\mu}{r}\|X\|_F^3.
	\]
\end{theorem}
\begin{proof}
	Let us consider the factorized function
	\[
	f(X)=F(XX^\top)=\frac12\langle\mathcal{A}(XX^\top),XX^\top\rangle+\langle B,XX^\top\rangle,\qquad X\in\mathbb{R}^{n\times r}.
	\]
	
	We have
	\begin{equation}\label{gradient}
		\nabla f(X)=2\bigl(\mathcal{A}(XX^\top)+B\bigr)X.
	\end{equation}
	
	Indeed, for any direction $V$ we have
	\[
	\langle\nabla f(X),V\rangle
	= \langle\mathcal{A}(XX^\top),XV^\top+VX^\top\rangle + \langle B,XV^\top+VX^\top\rangle.
	\]
	Since \(XX^\top \in \mathbb{S}^{n\times n}\), and the operator \(\mathcal{A}\) 
	maps symmetric matrices into symmetric ones, the matrix 
	\(\mathcal{A}(XX^\top)\) is symmetric. Consequently,
	\[
	\langle\mathcal{A}(XX^\top),XV^\top+VX^\top\rangle
	= 2\langle\mathcal{A}(XX^\top)X,V\rangle.
	\]
	Similarly, due to the symmetry of \(B\),
	\[
	\langle B,XV^\top+VX^\top\rangle = 2\langle BX,V\rangle.
	\]
	From here we obtain
	\[
	\langle\nabla f(X),V\rangle = 2\langle(\mathcal{A}(XX^\top)+B)X,V\rangle.
	\]
	
	Let us multiply \eqref{gradient} scalarly by $X$:
	\[
	\langle\nabla f(X),X\rangle
	= 2\bigl\langle\mathcal{A}(XX^\top)X,\,X\bigr\rangle + 2\bigl\langle BX,\,X\bigr\rangle.
	\]
	
	By the cyclic property of the trace,
	\[
	\langle\mathcal{A}(XX^\top)X,X\rangle
	= \mathrm{Tr}\bigl(X^\top\mathcal{A}(XX^\top)X\bigr)
	= \mathrm{Tr}\bigl(\mathcal{A}(XX^\top)XX^\top\bigr)
	= \langle\mathcal{A}(XX^\top),XX^\top\rangle.
	\]
	The same cyclic property gives, without any symmetry assumption on \(B\),
	\[
	\langle BX,X\rangle = \mathrm{Tr}(X^\top BX)=\mathrm{Tr}(BXX^\top)=\langle B,XX^\top\rangle.
	\]
	We note that the symmetry assumptions on \(\mathcal{A}(XX^\top)\) and \(B\) were used only
	in the preceding derivation of the gradient formula.
	
	Consequently,
	\begin{equation}\label{inner}
		\langle\nabla f(X),X\rangle = 2\Bigl(\langle\mathcal{A}(XX^\top),XX^\top\rangle + \langle B,XX^\top\rangle\Bigr).
	\end{equation}
	
	From the strong convexity of $F \left( Y \right)$ (\ref{strong-convex}) with $Y=XX^\top$ we have
	\[
	\langle\mathcal{A}(XX^\top),XX^\top\rangle \ge \mu\|XX^\top\|_F^2.
	\]
	Let us denote by $\sigma_1,\dots,\sigma_p$ the singular values of the matrix $X$, where \(p = r \).
	Then
	\[
	\|X\|_F^2 = \sum_{i=1}^p \sigma_i^2,\qquad
	\|XX^\top\|_F^2 = \|X^\top X\|_F^2 = \sum_{i=1}^p \sigma_i^4.
	\]
	
	\[
	\biggl( \sum_{i=1}^p a_i b_i \biggr)^{\!2}
	\le \biggl( \sum_{i=1}^p a_i^2 \biggr) \biggl( \sum_{i=1}^p b_i^2 \biggr),
	\qquad a_i, b_i \in \mathbb{R}.
	\]
	
	Set $a_i = 1$, $b_i = \sigma_i^2$ for \(i=1,\dots,p\).  Then
	\[
	\biggl( \sum_{i=1}^p 1 \cdot \sigma_i^2 \biggr)^{\!2}
	\le \biggl( \sum_{i=1}^p 1^2 \biggr) \biggl( \sum_{i=1}^p (\sigma_i^2)^2 \biggr)
	= p \cdot \sum_{i=1}^p \sigma_i^4 .
	\]
	Consequently,
	\[
	\sum_{i=1}^p \sigma_i^4 \;\ge\; \frac{1}{p}
	\biggl( \sum_{i=1}^p \sigma_i^2 \biggr)^{\!2}.
	\]
	
	\[
	\sqrt{\sum_{i=1}^p \sigma_i^4} \ge \frac{1}{\sqrt{p}}\sum_{i=1}^p \sigma_i^2
	\quad\Longrightarrow\quad
	\|XX^\top\|_F \ge \frac{1}{\sqrt{p}}\|X\|_F^2.
	\]
	After squaring we obtain
	\[
	\|XX^\top\|_F^2 \ge \frac{1}{p}\|X\|_F^4.
	\]
	Consequently,
	\begin{equation}\label{quad-term}
		\langle\mathcal{A}(XX^\top),XX^\top\rangle \ge \frac{\mu}{p}\|X\|_F^4.
	\end{equation}
	
	By Lemma~\ref{lem:norm_XX} it holds that \(\|XX^\top\|_F \le \|X\|_F^2\).
	
	Applying the Cauchy--Schwarz inequality and the estimate $\|XX^\top\|_F\le\|X\|_F^2$, we have
	\[
	|\langle B,XX^\top\rangle| \le \|B\|_F\|XX^\top\|_F \le \|B\|_F\|X\|_F^2,
	\]
	whence
	\begin{equation}\label{lin-term}
		\langle B,XX^\top\rangle \ge -\|B\|_F\|X\|_F^2.
	\end{equation}
	
	Substituting (\ref{quad-term}) and (\ref{lin-term}) into (\ref{inner}), we obtain
	\[
	\langle\nabla f(X),X\rangle \ge \frac{2\mu}{p}\|X\|_F^4 - 2\|B\|_F\|X\|_F^2.
	\]
	On the other hand, by the Cauchy--Schwarz inequality,
	\[
	\langle\nabla f(X),X\rangle \le \|\nabla f(X)\|_F\|X\|_F.
	\]
	Thus,
	\[
	\|\nabla f(X)\|_F\|X\|_F \ge \frac{2\mu}{p}\|X\|_F^4 - 2\|B\|_F\|X\|_F^2.
	\]
	If \(X = 0\), the desired inequality \eqref{final-lower} is trivially true.  
	Assume \(X \neq 0\). Dividing both sides by \(\|X\|_F > 0\), we obtain
	\[
	\|\nabla f(X)\|_F \ge \frac{2\mu}{p}\|X\|_F^3 - 2\|B\|_F\|X\|_F.
	\]
	This proves the theorem for all \(X \in \mathbb{R}^{n\times r}\).
\end{proof}

\begin{remark}[On the symmetry of $B$]
	\label{rem:sym-B}
	Under the conditions of Theorem~\ref{thm:lower-bound} the matrix \(B\) is assumed to be symmetric. This requirement is necessary for the correctness of the gradient formula of the factorized function
	\[
	\nabla f(X)=2\bigl(\mathcal{A}(XX^\top)+B\bigr)X .
	\]
	Indeed, if \(B\) is nonsymmetric, then the derivative of the term \(\langle B, XX^\top\rangle\) equals \((B+B^\top)X\), and not \(2BX\). Therefore, without the symmetry of the matrix \(B\) all the further computations using this gradient formula become incorrect.
	
	At the same time the equality itself
	\[
	\langle BX, X\rangle = \langle B, XX^\top\rangle
	\]
	holds for any square matrix \(B\) due to the cyclicity of the trace. 
	The symmetry of \(B\) is not required here.
	
	If in an applied problem the matrix \(B\) is nonsymmetric, it can be replaced, without changing the function \(F\), by its symmetric part
	\[
	B \mapsto \frac{B+B^\top}{2},
	\]
	since for a symmetric argument \(Y=XX^\top\) it holds that
	\[
	\langle B, Y\rangle = \left\langle \frac{B+B^\top}{2}, Y\right\rangle .
	\]
	After such a replacement the gradient formula takes the canonical form \(2(\mathcal{A}(XX^\top)+B)X\), and all the further results of the paper remain applicable.
\end{remark}

\begin{theorem}[Generalized smoothness of the function $f(X)=F(XX^\top)$]
	Let $F(Y)=\frac12\langle\mathcal{A}Y,Y\rangle + \langle B,Y\rangle$, where
	\begin{itemize}
		\item $\mathcal{A}\colon\mathbb{R}^{n\times n}\to\mathbb{R}^{n\times n}$ is a linear self-adjoint operator mapping $\mathbb{S}^{n\times n}$ into itself and satisfying the condition $\langle\mathcal{A}Y,Y\rangle \ge \mu\|Y\|_F^2$ for all $Y\in\mathbb{S}^{n\times n}$.
		
		\item $B\in\mathbb{S}^{n\times n}$ is a symmetric matrix.
	\end{itemize}
	Then the function $f(X)=F(XX^\top)$ satisfies
	the \emph{generalized $(p, L_0, L_p)$-smoothness condition} with exponent $p=2/3$ and constants
	$L_0, L_{2/3}$ defined in \eqref{L-constants}:
	\begin{equation}\label{final-p-L0Lp}
		\boxed{\;\|\nabla^2 f(X)\| \le L_0 \;+\; L_{2/3}\,\|\nabla f(X)\|_F^{2/3}\qquad \forall X\in\mathbb{R}^{n\times r}\;}
	\end{equation}
\end{theorem}
\begin{proof}
	Let us recall the key estimates.
	
	For a quadratic $F(Y)=\frac12\langle\mathcal{A}Y,Y\rangle+\langle B,Y\rangle$,
	direct computations (cf.~Proposition~2.1 in \cite{Dragomir21}) give
	\begin{equation}\label{Hess-upper}
		\|\nabla^2 f(X)\| \le 6\|\mathcal{A}\|_{\mathbb{S},\mathrm{op}}\,\|X\|_F^2 \;+\; 2\|B\|_F \qquad \forall X\in\mathbb{R}^{n\times r}.
	\end{equation}
	Here $\|\mathcal{A}\|_{\mathbb{S},\mathrm{op}}$ denotes the operator norm of the restriction of $\mathcal{A}$ to the subspace of symmetric matrices $\mathbb{S}^{n\times n}$, as defined in Remark~\ref{rem:norm-A}. This restricted norm is sufficient because in the Hessian of $f$ the operator $\mathcal{A}$ is applied only to symmetric matrices. The full operator norm on $\mathbb{R}^{n\times n}$ may be used as a conservative upper bound.
	(Everywhere below $\|X\|=\|X\|_F$.)
	
	Under the condition of strong convexity of $F$ with parameter $\mu>0$, in Theorem~\ref{thm:lower-bound}
	a lower bound on the gradient was established
	\begin{equation}\label{grad-lower}
		\|\nabla f(X)\|_F \ge \alpha\|X\|^3 \;-\; \beta\|X\|,
	\end{equation}
	where
	\[
	\alpha = \frac{2\mu}{ r }, \qquad \beta = 2\|B\|_F.
	\]
	
	Let us show that there exist explicit constants $C_1>0$ and $C_2\ge0$ for which
	\begin{equation}\label{X2-to-grad}
		\|X\|^2 \le C_1\,\|\nabla f(X)\|_F^{2/3} \;+\; C_2 \qquad \forall X\in\mathbb{R}^{n\times r}.
	\end{equation}
	
	Let us consider two cases depending on the magnitude of $\|X\|$.
	
	\smallskip\noindent\textit{Case 1 (small $X$).}
	If $\|X\| \le \sqrt{2\beta/\alpha}$, then trivially
	\[
	\|X\|^2 \le \frac{2\beta}{\alpha}.
	\]
	
	\smallskip\noindent\textit{Case 2 (large $X$).}
	Let $\|X\| > \sqrt{2\beta/\alpha}$. Then $\beta\|X\| < \frac{\alpha}{2}\|X\|^3$, and from \eqref{grad-lower} we obtain
	\[
	\|\nabla f(X)\|_F \ge \alpha\|X\|^3 - \frac{\alpha}{2}\|X\|^3 = \frac{\alpha}{2}\,\|X\|^3,
	\]
	whence
	\[
	\|X\|^3 \le \frac{2}{\alpha}\,\|\nabla f(X)\|_F,\qquad
	\|X\|^2 \le \Bigl(\frac{2}{\alpha}\Bigr)^{\!2/3}\|\nabla f(X)\|_F^{2/3}.
	\]
	
	\smallskip\noindent\textit{Combination of the cases.}
	For any $X$ it holds that
	\[
	\|X\|^2 \le \max\!\Bigl\{\frac{2\beta}{\alpha},\; \Bigl(\frac{2}{\alpha}\Bigr)^{\!2/3}\|\nabla f(X)\|_F^{2/3}\Bigr\}.
	\]
	Since both arguments of the maximum are nonnegative, $\max\{a,b\}\le a+b$, consequently,
	\[
	\|X\|^2 \le \Bigl(\frac{2}{\alpha}\Bigr)^{\!2/3}\|\nabla f(X)\|_F^{2/3} \;+\; \frac{2\beta}{\alpha}.
	\]
	Thus, one can set
	\begin{equation}\label{C1C2}
		C_1 = \left(\frac{2}{\alpha}\right)^{2/3}
		= \left(\frac{ r }{\mu}\right)^{2/3},
		\qquad
		C_2 = \frac{2\beta}{\alpha}
		= \frac{2 r \|B\|_F}{\mu}.
	\end{equation}
	
	Substituting the estimate \eqref{X2-to-grad} with the constants \eqref{C1C2} into \eqref{Hess-upper}, we find
	\[
	\|\nabla^2 f(X)\|
	\le 6\|\mathcal{A}\|_{\mathbb{S},\mathrm{op}}\bigl(C_1\|\nabla f(X)\|_F^{2/3} + C_2\bigr) \;+\; 2\|B\|_F
	= L_{2/3}\,\|\nabla f(X)\|_F^{2/3} \;+\; L_0,
	\]
	where
	\begin{equation}\label{L-constants}
		\begin{aligned}
			L_{2/3}
			&= 6\|\mathcal{A}\|_{\mathbb{S},\mathrm{op}}\,C_1
			= 6\|\mathcal{A}\|_{\mathbb{S},\mathrm{op}}
			\left(\frac{r}{\mu}\right)^{2/3},\\[4pt]
			L_0
			&= 6\|\mathcal{A}\|_{\mathbb{S},\mathrm{op}}\,C_2 + 2\|B\|_F
			= 6\|\mathcal{A}\|_{\mathbb{S},\mathrm{op}}\,
			\frac{2 r \|B\|_F}{\mu}
			+ 2\|B\|_F.
		\end{aligned}
	\end{equation}
\end{proof}

\begin{lemma}[Estimate of the norm of the difference of gradients]
	\label{lem:grad-diff}
	Let \(f\colon \mathbb{R}^{n\times r}\to\mathbb{R}\) be twice continuously differentiable.
	Assume that there exist constants
	\[
	L_0\ge 0,
	\qquad
	L_{2/3}\ge 0,
	\qquad
	L_0+L_{2/3}>0,
	\]
	such that \(f\) satisfies the generalized
	\((2/3,L_0,L_{2/3})\)-smoothness condition \eqref{def_gensmooth}, i.e.
	\[
	\left\|\nabla^2 f(Z)\right\|_{\mathrm{op}}
	\le
	L_0+L_{2/3}\left\|\nabla f(Z)\right\|_F^{2/3}
	\qquad
	\forall Z\in\mathbb{R}^{n\times r}.
	\]
	Then for any \(X,Y\in\mathbb{R}^{n\times r}\),
	\[
	\left\|\nabla f(Y)-\nabla f(X)\right\|_F
	\le
	q^{-1}\Bigl(
	\left\|Y-X\right\|_F;\,
	\left\|\nabla f(X)\right\|_F
	\Bigr),
	\]
	where
	\[
	q(s;a)
	=
	\int_0^s \frac{dv}{\ell(a+v)},
	\qquad
	\ell(g)=L_0+L_{2/3}g^{2/3}.
	\]
\end{lemma}

\begin{proof}
	If \(X=Y\), then the statement is trivial, since \(q^{-1}(0;a)=0\).
	Further we assume \(X\ne Y\).
	
	Let us use a technical trick for the proof: 
	for \(\varepsilon > 0\) define the regularized function
	\[
	\ell_\varepsilon(g) = L_0 + L_{2/3}(g + \varepsilon)^{2/3}, \qquad g\ge0.
	\]
	It is continuous, positive and nondecreasing.

	Denote \(a=\|\nabla f(X)\|_F\),
	\(U=(Y-X)/\|Y-X\|_F\),
	\(h(t)=X+tU\). Then \(\|U\|_F=1\).
	Set \(v(t)=\|\nabla f(h(t))-\nabla f(X)\|_F\).
	It is clear that \(v(0)=0\).
	
	Since \(f\in C^2\), by the fundamental theorem of calculus,
	\[
	\nabla f(h(t))-\nabla f(X)
	=
	\int_0^t \nabla^2 f(h(\tau))[U]\,d\tau.
	\]
	Therefore, by the triangle inequality and since \(\|U\|_F=1\),
	\[
	v(t)
	\le
	\int_0^t
	\left\|\nabla^2 f(h(\tau))\right\|_{\mathrm{op}}
	\,d\tau.
	\]
	
	Due to the \((2/3,L_0,L_{2/3})\)-smoothness of \(f\) and the triangle inequality we have
	\[
	v(t)\le \int_0^t \|\nabla^2 f(h(\tau))\|_{\mathrm{op}} \|U\|_F\,d\tau
	\le \int_0^t \ell_\varepsilon(\|\nabla f(h(\tau))\|_F)\,d\tau
	\le \int_0^t \ell_\varepsilon(a+v(\tau))\,d\tau,
	\]
	where we used that \(\|\nabla f(h(\tau))\|_F\le a+v(\tau)\)
	and \(\ell_\varepsilon\) is nondecreasing, as well as that
	\(\|\nabla^2 f(X)\|_{\mathrm{op}} \le \ell(\|\nabla f(X)\|_F)
	\le \ell_\varepsilon(\|\nabla f(X)\|_F)\) for all \(X\).
	
	Set \(\phi(t)=\int_0^t \ell_\varepsilon(a+v(\tau))\,d\tau\).
	Then \(v(t)\le \phi(t)\), and
	\[
	\phi'(t)=\ell_\varepsilon(a+v(t))\le \ell_\varepsilon(a+\phi(t))
	\]
	due to the monotonicity of \(\ell_\varepsilon\). Thus,
	\[
	\frac{\phi'(t)}{\ell_\varepsilon(a+\phi(t))} \le 1.
	\]
	
	Let us introduce the function
	\[
	F_\varepsilon(r) = q_\varepsilon(r;a) = \int_0^r \frac{dv}{\ell_\varepsilon(a+v)}.
	\]
	Then \(F_\varepsilon'(s) = 1/\ell_\varepsilon(a+s)>0\), therefore \(F_\varepsilon\)
	is strictly increasing. Note that
	\[
	\frac{d}{dt} F_\varepsilon(\phi(t))
	= \frac{\phi'(t)}{\ell_\varepsilon(a+\phi(t))}
	\le 1.
	\]
	Integrating over \(t\) and taking into account that \(F_\varepsilon(\phi(0))=F_\varepsilon(0)=0\),
	we obtain
	\[
	F_\varepsilon(\phi(t)) \le t.
	\]
	
	On the other hand, let \(g(t)=F_\varepsilon^{-1}(t)\). Then
	\[
	F_\varepsilon(g(t))=t,
	\]
	and, differentiating this equality,
	\[
	\frac{g'(t)}{\ell_\varepsilon(a+g(t))}=1
	\quad\Longrightarrow\quad
	g'(t)=\ell_\varepsilon(a+g(t)),
	\]
	that is, \(g\) is a solution of the corresponding differential equation.
	Since \(F_\varepsilon\) is strictly increasing, from
	\(F_\varepsilon(\phi(t)) \le t = F_\varepsilon(g(t))\)
	it follows that
	\[
	\phi(t) \le g(t) = F_\varepsilon^{-1}(t) = q_\varepsilon^{-1}(t;a).
	\]
	Thus,
	\[
	v(t) \le q_\varepsilon^{-1}(t;a).
	\]
	
	For \(t=\|Y-X\|_F\) this gives
	\[
	\|\nabla f(Y)-\nabla f(X)\|_F \le q_\varepsilon^{-1}\bigl(\|Y-X\|_F;a\bigr).
	\]

	It remains to pass to the limit \(\varepsilon\to0^+\).  
	Since
	\[
	\ell_\varepsilon(g) \downarrow \ell(g) = L_0 + L_{2/3} g^{2/3}
	\qquad \text{as } \varepsilon \downarrow 0
	\quad \text{for every } g \ge 0,
	\]
	we have
	\[
	\frac{1}{\ell_\varepsilon(g)} \uparrow \frac{1}{\ell(g)}
	\qquad \text{for every } g > 0.
	\]
	By the monotone convergence theorem,
	\[
	q_\varepsilon(s; a) \uparrow q(s; a)
	= \int_0^s \frac{dv}{L_0 + L_{2/3}(a + v)^{2/3}}
	\qquad \forall s \ge 0.
	\]
	Since \(q_{\max}(a) = \int_0^\infty \frac{dv}{L_0 + L_{2/3}(a + v)^{2/3}} = \infty\),
	the functions \(q_\varepsilon(\cdot;a)\) and \(q(\cdot;a)\) are defined on the whole 
	\([0,\infty)\) and are strictly increasing. From the monotone convergence 
	\(q_\varepsilon\uparrow q\) the pointwise convergence of the inverse functions follows:
	\[
	q_\varepsilon^{-1}(\tau; a)\downarrow q^{-1}(\tau; a)
	\qquad \forall \tau\ge0.
	\]
	Indeed, for any \(\tau\ge0\) set 
	\(r_\varepsilon = q_\varepsilon^{-1}(\tau;a)\). Then 
	\(q_\varepsilon(r_\varepsilon)=\tau\), and due to the monotonicity 
	\(q_\varepsilon\uparrow q\) we obtain \(q(r_\varepsilon)\ge\tau\), 
	whence \(r_\varepsilon\ge q^{-1}(\tau;a)\). On the other hand, for 
	any \(r > q^{-1}(\tau;a)\) it holds that \(q(r)>\tau\), therefore for 
	sufficiently small \(\varepsilon\) we have \(q_\varepsilon(r) > \tau\), 
	which gives \(r_\varepsilon \le r\). Thus, 
	\(\lim_{\varepsilon\to0^+} q_\varepsilon^{-1}(\tau;a) = q^{-1}(\tau;a)\).
	
	Passing to the limit in the inequality
	\[
	\|\nabla f(Y)-\nabla f(X)\|_F \le q_\varepsilon^{-1}\bigl(\|Y-X\|_F;a\bigr),
	\]
	whose left-hand side does not depend on \(\varepsilon\),
	we obtain the required statement.
\end{proof}

\begin{remark}[On the case \(L_0=0\)]
	The proof of Lemma~\ref{lem:grad-diff} uses regularization and does not require
	local Lipschitz continuity of \(\ell\), therefore it remains valid also for
	\(L_0=0\), \(L_{2/3}>0\). 
	No separate consideration is required.
\end{remark}

\subsection{Key Inequalities for \(\ell( g ) = L_0 + L_{2/3} g^{2/3}\)}

The key tool for the analysis of gradient descent for $\ell$-smooth functions in the work~\cite{tyurin2025toward} consists of two lemmas: the estimate of the gradient increment (Lemma~4.3) and the estimate of the function increment (Lemma~4.5). Their standard proof relies on the assumption of local Lipschitz continuity of the function~$\ell$ (see Assumption~3.1 in~\cite{tyurin2025toward}). However, for the function
\[
\ell(g)=L_0 + L_{2/3} g^{2/3},
\]
arising in our matrix factorization problem, local Lipschitz continuity fails at zero: the derivative $\ell'(g)\sim g^{-1/3}$ is unbounded as $g \to 0^+$. Nevertheless, a detailed analysis shows that in none of the steps of the proofs is local Lipschitz continuity actually used --- continuity, strict positivity and monotone nondecrease of~$\ell$ suffice.

In order to construct rigorous proofs and to justify explicitly the applicability of the mentioned estimates to our case, we give below complete proofs of both lemmas, following the logic of~\cite{tyurin2025toward}, but with relaxed requirements on~$\ell$. This simultaneously demonstrates that Tyurin's results can be extended to a wider class of functions, and serves as a foundation for all the subsequent constructions of algorithms and the analysis of their convergence.

\section{The \texorpdfstring{$q$}{q}-transform of the Smoothness Modulus and Integral Increment Estimates}
Let us consider the function \( \ell \colon[0,\infty) \to [0,\infty) \) given by the formula
\[
\ell( g ) = L_0 + L_{2/3} g^{2/3},
\]
where \(L_0 \ge 0\), \(L_{2/3} \ge 0\), and moreover \(L_0 + L_{2/3} > 0\). 
The function \(\ell\) is continuous, nonnegative and nondecreasing; 
it is positive on \((0,\infty)\).

For an arbitrary \(a\ge 0\) let us define the function
\[
q(s; a) = \int_{0}^{s} \frac{dv}{\ell(a + v)},\qquad s \ge 0.
\]
Since \(\ell(a+v)>0\) for all \(v>0\), and for \(L_0=0\) the singularity at zero
has the form \(\frac{1}{L_{2/3}v^{2/3}}\), the integral is finite for any finite \( s \ge 0\).
Moreover, \(q(\cdot\,;a)\) is strictly increasing, continuously differentiable and
\(q'(s; a) = 1/\ell(a + s)>0\) for all \(s > 0\).

Set
\[
q_{\max}(a)=\int_{0}^{\infty}\frac{dv}{\ell(a+v)}.
\]
For \(L_{2/3}=0\) we have \(q_{\max}(a)=\infty\) because \(\ell(a+v)=L_0\) is constant.
For \(L_{2/3}>0\), since \(\ell(a+v) = L_0 + L_{2/3}(a+v)^{2/3} \le C v^{2/3}\) for all sufficiently large \(v\) and some constant \(C>0\), the integrand \(1/\ell(a+v)\) is bounded below by a multiple of \(v^{-2/3}\), whose integral over \([0,\infty)\) diverges. Therefore in this case as well \(q_{\max}(a)=\infty\). Consequently, for any
\(a, s \ge 0\) the quantity \(q(s; a)\) is defined and invertible; the inverse function
is denoted by \(q^{-1}(\tau; a)\) \( \left( \tau \ge 0, q^{-1}(0; a) = 0 \right) \).

Here $\tau$ is the distance between the points, $s$ is the change of the gradient from the initial point to some point.
$q \left( s; a \right)$ can be interpreted as the distance that one has to travel in order to accumulate such a change of the gradient $s$.
$a$ is the value (magnitude) of the gradient at the initial point.
$q^{-1} \left( \tau; a \right)$ maps the distance between the points into the maximal possible magnitude of the change of the gradient.
$g$ is the magnitude of the gradient norm at the point under consideration.

\begin{remark}[Convention for zero gradient and early stopping]
	\label{def:zero_gradient_convention}
	For the function
	\[
	\ell(g)=L_0+L_{2/3}g^{2/3},
	\]
	define, for any \(a\ge 0\),
	\[
	\Phi(a)
	:=
	\frac{a^2}{\ell(2a)}
	=
	\frac{a^2}{L_0+2^{2/3}L_{2/3}a^{2/3}}.
	\]
	We set
	\[
	\Phi(0):=0.
	\]
	Thus, whenever \(a=0\), the ratio
	\[
	\frac{a^2}{L_0+2^{2/3}L_{2/3}a^{2/3}}
	\]
	is understood as \(0\), including the case \(L_0=0\).
	
	In particular, in all convergence statements below, every expression of the form
	\[
	\frac{\|\nabla f(X^t)\|_F^2}
	{L_0+2^{2/3}L_{2/3}\|\nabla f(X^t)\|_F^{2/3}}
	\]
	is interpreted as
	\[
	\Phi\bigl(\|\nabla f(X^t)\|_F\bigr).
	\]
	
	Moreover, if an algorithm stops at some iteration \(\tau\) because
	\[
	\|\nabla f(X^\tau)\|_F=0,
	\]
	then for any horizon \(T>\tau\) we extend the generated sequence by setting
	\[
	X^t:=X^\tau
	\qquad
	\forall\,t\ge \tau.
	\]
	Consequently, minima over \(0\le t<T\) are taken over this extended sequence, and the corresponding ratio is equal to zero for all \(t\ge \tau\).
\end{remark}

\begin{lemma}[Integral estimate of the function increment]
	\label{lem:function_increment}
	Let \(f\colon\mathbb{R}^{n\times r}\to\mathbb{R}\) be twice continuously differentiable and satisfy the
	\((2/3,L_0,L_{2/3})\)-smoothness condition \eqref{def_gensmooth} with
	\[
	\ell(g)=L_0+L_{2/3}g^{2/3},
	\qquad
	L_0\ge0,\; L_{2/3}\ge0,\; L_0+L_{2/3}>0.
	\]
	Then for any \(X,Y\in\mathbb{R}^{n\times r}\) it holds that
	\[
	f(Y)
	\le
	f(X)
	+
	\inprod{\nabla f(X)}{Y-X}
	+
	\int_0^{\norm{Y-X}_F}
	q^{-1}\bigl(\tau;\norm{\nabla f(X)}_F\bigr)\,d\tau,
	\]
	where \(q(s;a)=\int_0^s \frac{dv}{\ell(a+v)}\).
\end{lemma}

\begin{proof}
	If \(X=Y\), the statement is trivial. Assume \(X\neq Y\) and denote
	\[
	a=\norm{\nabla f(X)}_F,\qquad
	R=\norm{Y-X}_F>0.
	\]
	Let us consider the line segment
	\[
	h(t)=X+t(Y-X),\qquad t\in[0,1].
	\]
	By the fundamental theorem of calculus,
	\[
	f(Y)-f(X)
	=
	\int_0^1
	\inprod{\nabla f(h(t))}{Y-X}\,dt.
	\]
	We rewrite the integrand as
	\[
	\inprod{\nabla f(h(t))}{Y-X}
	=
	\inprod{\nabla f(X)}{Y-X}
	+
	\inprod{\nabla f(h(t))-\nabla f(X)}{Y-X}.
	\]
	By the Cauchy--Schwarz inequality,
	\[
	\inprod{\nabla f(h(t))-\nabla f(X)}{Y-X}
	\le
	\norm{\nabla f(h(t))-\nabla f(X)}_F\cdot \norm{Y-X}_F.
	\]
	From Lemma~\ref{lem:grad-diff} (the gradient difference estimate) we have
	\[
	\norm{\nabla f(h(t))-\nabla f(X)}_F
	\le
	q^{-1}\bigl(t\norm{Y-X}_F;\, a\bigr).
	\]
	Therefore
	\[
	\inprod{\nabla f(h(t))}{Y-X}
	\le
	\inprod{\nabla f(X)}{Y-X}
	+
	q^{-1}\bigl(tR;\, a\bigr)\cdot R.
	\]
	Integrating over \(t\in[0,1]\), we obtain
	\[
	f(Y)-f(X)
	\le
	\inprod{\nabla f(X)}{Y-X}
	+
	R\int_0^1 q^{-1}(tR;\,a)\,dt.
	\]
	The change of variable \(\tau=tR\) gives
	\[
	R\int_0^1 q^{-1}(tR;\,a)\,dt
	=
	\int_0^R q^{-1}(\tau;\,a)\,d\tau.
	\]
	Thus
	\[
	f(Y)
	\le
	f(X)
	+
	\inprod{\nabla f(X)}{Y-X}
	+
	\int_0^{\norm{Y-X}_F} q^{-1}\bigl(\tau;\norm{\nabla f(X)}_F\bigr)\,d\tau,
	\]
	which is exactly the claimed inequality.
\end{proof}

\section{Practical Implementation of Gradient Descent Method}

In this section we describe a practical implementation of the gradient descent method for minimizing the function \(f\) under the nonnegativity constraint \(X\ge 0\). We rely on the basic adaptive projected gradient descent (Algorithm~\ref{alg:adaptive_gd_tyurin}), in which the step size is chosen in accordance with the integral estimate of the function increment and the generalized smoothness modulus
\[
\ell(g)=L_0+L_{2/3}g^{2/3}.
\]
The method uses the \(q\)-function to determine a theoretically justified step and applies the projection onto the nonnegative orthant. We then discuss modifications that allow the step size to be increased without losing the theoretical convergence guarantees.

\begin{algorithm}[H]
	\caption{Basic adaptive projected GD}
	\label{alg:adaptive_gd_tyurin}
	\Input{\(X^0 \ge 0\), \(L_0, L_{2/3}\), \(T\)}
	\Output{\(X^T\)}
	\For{$t=0,1,\dots,T-1$}{
		$g_t \gets \nabla f(X^t)$\;
		$a_t \gets \norm{g_t}_F$\;
		\If{$a_t = 0$}{
			\Return $X^t$\;
		}
		$\eta_t \gets \frac{1}{a_t} q(a_t;\, a_t)$\;
		$X^{t+1} \gets \Proj\bigl(X^t - \eta_t\, g_t\bigr)$\;
	}
\end{algorithm}

\subsection{Motivation}
The theoretical step $\eta_t^{\rm theory} = \frac{1}{a_t} q(a_t;\, a_t)$, defined in Algorithm~\ref{alg:adaptive_gd_tyurin},
minimizes the upper bound from the integral lemma on the function increment.
However, this upper bound may be substantially inexact:
in practice the real descent of the function along the gradient direction can often be performed
with a significantly larger step than the theory guarantees.
In this section we propose a modification which
allows one to use an increased step while preserving the theoretical
convergence guarantees.

\subsection{Description of the Modifications}
The proposed method contains five modifications of the basic adaptive gradient descent (Algorithm~\ref{alg:adaptive_gd_tyurin}):

\begin{enumerate}
	\item \textbf{Increase of the initial step.}
	Instead of the theoretical step $\frac{1}{a_t} q(a_t;\, a_t)$ the
	increased step
	\[
	\eta_t^{(0)} = \alpha_t \cdot \frac{1}{a_t} q(a_t;\, a_t),
	\qquad \alpha_t \ge 1,
	\]
	is used, where $\alpha_t$ is the enlargement factor, adapted
	at every iteration.
	
	\item \textbf{Armijo condition (sufficient descent).}
	A trial step $\eta$ is accepted only if
	\[
	f\bigl(X - \eta\, \nabla f(X)\bigr)
	\;\le\;
	f(X) - c\,\eta\, \|\nabla f(X)\|_F^2,
	\]
	where $c \in (0, 1/2)$ is the sufficient descent parameter.
	This condition guarantees that the increased step
	indeed yields a decrease of the function.
	
	\item \textbf{Backtracking.}
	If the Armijo condition is not satisfied, the step is successively
	halved:
	\[
	\eta \leftarrow \eta / 2
	\]
	until the condition is satisfied or the iteration limit is exhausted.
	
	\item \textbf{Adaptation of the factor $\alpha_t$.}
	After every iteration the factor is updated according to the rule:
	\[
	\alpha_{t+1} =
	\begin{cases}
		\min\bigl\{\alpha_{\max},\; \alpha_t \cdot \rho_+\bigr\},
		& \text{if backtracking was not needed},\\[6pt]
		\max\bigl\{\alpha_{\min},\; \alpha_t \cdot \rho_-\bigr\},
		& \text{if backtracking was needed},
	\end{cases}
	\]
	where $\rho_+ > 1$ (for example, $\rho_+ = 1{.}1$),
	$\rho_- < 1$ (for example, $\rho_- = 0{.}8$).
	The logic: if the increased step passed without divisions ---
	the theoretical step is too conservative, one can increase $\alpha$;
	if divisions were needed --- the step was too large,
	one should decrease $\alpha$.
	
	\item \textbf{Range restriction.}
	The factor is bounded:
	\[
	\alpha_t \in [\alpha_{\min},\, \alpha_{\max}]
	= [1,\, 5].
	\]
	The lower bound $\alpha_{\min} = 1$ guarantees that the step
	never becomes smaller than the theoretical one.
	The upper bound $\alpha_{\max} = 5$ prevents
	excessively aggressive steps.
\end{enumerate}

\begin{remark}[On the number of backtracking iterations]
	Since $\eta$ is halved at every iteration
	of the backtracking and cannot become smaller than $\eta_t^{\rm theory}$
	(for which the descent is guaranteed theoretically),
	the number of divisions does not exceed
	\[
	\left\lceil \log_2 \alpha_t \right\rceil
	\le \left\lceil \log_2 \alpha_{\max} \right\rceil
	= \lceil \log_2 5 \rceil = 3.
	\]
	Thus, at every iteration at most
	3--4 additional evaluations of the function $f$ are required.
\end{remark}

\begin{algorithm}[H]
	\caption{Adaptive GD with backtracking for the unconstrained problem}
	\label{alg:adaptive_gd_backtrack}
	\Input{\(X^0\), \(L_0, L_{2/3}\), \(\alpha_0 \ge 1\),
		\(\alpha_{\min} = 1\), \(\alpha_{\max} = 5\),
		\(\rho_+ = 1{.}1\), \(\rho_- = 0{.}8\),
		\(c \in (0, 1/2)\), \(T\)}
	\Output{\(X^T\)}
	$\alpha \gets \alpha_0$\;
	\For{$t=0,1,\dots,T-1$}{
		$g_t \gets \nabla f(X^t)$\;
		$a_t \gets \norm{g_t}_F$\;
		\If{$a_t = 0$}{
			\Return $X^t$\;
		}
		$\eta_t^{\rm theory} \gets \frac{1}{a_t} q(a_t;\, a_t)$\;
		$\eta \gets \alpha \cdot \eta_t^{\rm theory}$\;
		$f_t \gets f(X^t)$\;
		${\rm backtrack} \gets {\bf false}$\;
		\While{$f(X^t - \eta\, g_t) > f_t - c\,\eta\, a_t^2$
			\textnormal{and}
			$\eta/2 \ge \eta_t^{\rm theory}$}{
			$\eta \gets \eta / 2$\;
			${\rm backtrack} \gets {\bf true}$\;
		}
		\If{$f(X^t - \eta\, g_t) > f_t - c\,\eta\, a_t^2$}{
			$\eta \gets \eta_t^{\rm theory}$\;
			${\rm backtrack} \gets {\bf true}$\;
		}
		\If{${\rm backtrack}$}{
			$\alpha \gets \max(\alpha_{\min},\; \alpha \cdot \rho_-)$\;
		}
		\Else{
			$\alpha \gets \min(\alpha_{\max},\; \alpha \cdot \rho_+)$\;
		}
		$X^{t+1} \gets X^t - \eta\, g_t$\;
	}
\end{algorithm}

\begin{lemma}[Descent along a direction]
	\label{lem:descent_direction}
	Let
	\[
	g=\nabla f(X),
	\qquad
	a=\norm{g},
	\]
	\(d\) be a unit direction, \(\norm{d}=1\), and
	\[
	b=\inprod{g}{d}\ge0.
	\]
	If
	\[
	\eta=q(b;a),
	\]
	then
	\[
	f(X-\eta d)
	\le
	f(X)
	-
	\int_0^b
	\frac{b-s}{\ell(a+s)}\,ds.
	\]
\end{lemma}

\begin{proof}
	Set
	\[
	Y=X-\eta d.
	\]
	Then
	\[
	\norm{Y-X}=\eta,
	\qquad
	\inprod{g}{Y-X}=-\eta b.
	\]
	By the integral lemma,
	\[
	f(Y)
	\le
	f(X)-\eta b
	+
	\int_0^\eta q^{-1}(\tau;a)\,d\tau.
	\]
	Since \(\eta=q(b;a)\), let us make the change of variable
	\[
	\tau=q(s;a),
	\qquad
	s\in[0,b].
	\]
	Then
	\[
	d\tau=\frac{ds}{\ell(a+s)},
	\qquad
	q^{-1}(\tau;a)=s.
	\]
	Consequently,
	\[
	\int_0^\eta q^{-1}(\tau;a)\,d\tau
	=
	\int_0^b\frac{s}{\ell(a+s)}\,ds.
	\]
	Moreover,
	\[
	\eta b
	=
	b\int_0^b\frac{ds}{\ell(a+s)}
	=
	\int_0^b\frac{b}{\ell(a+s)}\,ds.
	\]
	Therefore
	\[
	-\eta b+\int_0^\eta q^{-1}(\tau;a)\,d\tau
	=
	-
	\int_0^b
	\frac{b-s}{\ell(a+s)}\,ds.
	\]
	The lemma is proved.
\end{proof}

\begin{lemma}[Descent guarantee]
	\label{lem:armijo_descent}
	Suppose that at iteration $t$ of Algorithm~\ref{alg:adaptive_gd_backtrack}
	the step $\eta_t$ was performed. Then
	\[
	f(X^{t+1}) \le f(X^t) - c\,\eta_t\, \norm{\nabla f(X^t)}_F^2.
	\]
\end{lemma}

\begin{proof}
	Let \(a_t=\|\nabla f(X^t)\|_F\).
	If the Armijo condition is satisfied for the current trial step,
	then the required inequality holds by construction.
	
	Otherwise the algorithm sets
	\[
	\eta_t = \eta_t^{\rm theory}
	= \frac{1}{a_t} q(a_t; a_t).
	\]
	We now show that this theoretical step satisfies the Armijo condition.
	
	Let
	\[
	d_t = \frac{g_t}{a_t},
	\qquad
	s_t = q(a_t; a_t).
	\]
	Then \(\|d_t\|_F=1\), and the step multiplier is
	\[
	\eta_t^{\rm theory}
	=
	\frac{s_t}{a_t}
	=
	\frac{q(a_t; a_t)}{a_t}.
	\]
	Hence the actual step can be written as
	\[
	X^t - \eta_t^{\rm theory} g_t
	=
	X^t - \frac{s_t}{a_t} g_t
	=
	X^t - s_t d_t.
	\]
	
	By Lemma~\ref{lem:descent_direction},
	for the step length \(s_t = q(a_t;a_t)\) in the direction \(d_t\)
	we have
	\[
	f\bigl(X^t - s_t d_t\bigr)
	\le
	f(X^t)
	-
	\int_0^{a_t}
	\frac{a_t-u}{\ell(a_t+u)}\,du.
	\]

	Let us introduce the two functions
	\[
	\phi(u) = a_t - u,
	\qquad
	h(u) = \frac{1}{\ell(a_t+u)},
	\qquad
	u \in [0, a_t].
	\]
	Both functions are nonincreasing on this interval: the function \(\phi\) decreases trivially (\(\phi'(u) = -1 < 0\)), and the function \(h\) decreases because \(\ell\) is nondecreasing by assumption, and therefore its reciprocal \(1/\ell\) is nonincreasing.
	
	Recall the classical Chebyshev integral inequality: if two integrable functions \(\phi\) and \(h\) are either both nonincreasing or both nondecreasing on \([0, L]\), then their synchronous monotonicity implies
	\[
	\int_0^L \phi(u)\,h(u)\,du
	\;\ge\;
	\frac{1}{L}
	\left(\int_0^L \phi(u)\,du\right)
	\left(\int_0^L h(u)\,du\right).
	\]
	Applying this inequality with \(L = a_t\) to our functions \(\phi\) and \(h\), we obtain
	\[
	\int_0^{a_t}
	\frac{a_t-u}{\ell(a_t+u)}\,du
	\;\ge\;
	\frac{1}{a_t}
	\left(\int_0^{a_t}(a_t-u)\,du\right)
	\left(\int_0^{a_t}\frac{du}{\ell(a_t+u)}\right).
	\]
	
	The first integral equals \(a_t^2/2\), and the second equals
	\(s_t = q(a_t;a_t)\). Hence
	\[
	\int_0^{a_t}
	\frac{a_t-u}{\ell(a_t+u)}\,du
	\ge
	\frac{a_t}{2}\, q(a_t;a_t)
	=
	\frac{a_t s_t}{2}.
	\]
	
	Since
	\[
	\eta_t^{\rm theory} a_t^2
	=
	\frac{s_t}{a_t}\, a_t^2
	=
	s_t a_t,
	\]
	we obtain
	\[
	\int_0^{a_t}
	\frac{a_t-u}{\ell(a_t+u)}\,du
	\ge
	\frac{1}{2}\,\eta_t^{\rm theory} a_t^2.
	\]
	Because \(c < 1/2\), it follows that
	\[
	\int_0^{a_t}
	\frac{a_t-u}{\ell(a_t+u)}\,du
	>
	c\,\eta_t^{\rm theory} a_t^2.
	\]
	Thus the theoretical step satisfies the Armijo condition, and the lemma
	is proved.
\end{proof}

\begin{theorem}[Convergence of the method with backtracking]
	\label{teo:backtrack_convergence}
	Let \(f\) be bounded from below:
	\[
	f(X)\ge f^*,
	\]
	and satisfy the \((2/3,L_0,L_{2/3})\)-smoothness condition, where
	\[
	L_0\ge0,\qquad L_{2/3}\ge0,\qquad L_0+L_{2/3}>0.
	\]
	Then for the sequence generated by
	Algorithm~\ref{alg:adaptive_gd_backtrack},
	with the convention of Definition~\ref{def:zero_gradient_convention},
	for any \(T\ge 1\) it holds that
	\[
	c\sum_{t=0}^{T-1}
	\frac{\|\nabla f(X^t)\|_F^2}
	{L_0 + 2^{2/3}L_{2/3}\|\nabla f(X^t)\|_F^{2/3}}
	\le
	\Delta.
	\]
	Consequently,
	\[
	\min_{0 \le t < T}
	\frac{\|\nabla f(X^t)\|_F^2}
	{L_0 + 2^{2/3}L_{2/3}\|\nabla f(X^t)\|_F^{2/3}}
	\le
	\frac{\Delta}{c\,T},
	\]
	where
	\[
	\Delta = f(X^0) - f^*.
	\]
\end{theorem}

\begin{proof}
	If for some \(\tau<T\) we have
	\[
	\|\nabla f(X^\tau)\|_F=0,
	\]
	then, by Definition~\ref{def:zero_gradient_convention},
	the algorithm stops and the sequence is extended by stationarity.
	All summands with \(t\ge\tau\) are zero.
	The finite sum over the iterations before \(\tau\) is bounded by the same
	telescoping argument applied up to the stopping time, and the minimum is zero.
	Hence it suffices to give the proof for the case where
	\[
	\|\nabla f(X^t)\|_F>0
	\qquad
	\forall\,t=0,\dots,T-1.
	\]
	
	Therefore, in the proof below we may assume that
	\[
	a_t:=\|\nabla f(X^t)\|_F>0
	\qquad
	\forall\,t=0,\dots,T-1.
	\]
	
	By Lemma~\ref{lem:armijo_descent},
	\[
	f(X^{t+1}) \le f(X^t) - c\,\eta_t\, a_t^2,
	\]
	where $a_t = \|\nabla f(X^t)\|_F$ and $\eta_t$ is the step multiplier chosen by the algorithm.
	
	Since $\alpha_t \ge \alpha_{\min} = 1$, the algorithm never selects a step smaller than the theoretical one:
	\[
	\eta_t \ge \eta_t^{\mathrm{theory}}
	= \frac{1}{a_t}\, q(a_t; a_t).
	\]
	
	We now lower-bound this quantity. Since $\ell$ is nondecreasing,
	for all $s \in [0, a_t]$ we have
	\[
	\ell(a_t + s) \le \ell(2a_t).
	\]
	Hence
	\[
	q(a_t; a_t)
	=
	\int_0^{a_t} \frac{ds}{\ell(a_t + s)}
	\ge
	\int_0^{a_t} \frac{ds}{\ell(2a_t)}
	=
	\frac{a_t}{\ell(2a_t)}.
	\]
	Dividing by $a_t > 0$, we obtain
	\[
	\frac{1}{a_t} q(a_t; a_t)
	\ge
	\frac{1}{\ell(2a_t)}.
	\]
	Therefore
	\[
	\eta_t \ge \frac{1}{\ell(2a_t)}.
	\]
	
	Substituting this into the descent inequality gives
	\[
	f(X^{t+1})
	\le f(X^t) - \frac{c\, a_t^2}{\ell(2a_t)}.
	\]
	
	Summing over $t = 0, \dots, T-1$:
	\[
	c \sum_{t=0}^{T-1} \frac{a_t^2}{\ell(2a_t)}
	\le f(X^0) - f(X^T) \le \Delta.
	\]
	The minimum does not exceed the average:
	\[
	\min_{0 \le t < T} \frac{a_t^2}{\ell(2a_t)}
	\le \frac{\Delta}{c\,T}.
	\]
	Since $\ell(2a_t) = L_0 + L_{2/3}(2a_t)^{2/3} = L_0 + 2^{2/3}L_{2/3} a_t^{2/3}$,
	this is exactly the claimed inequality.
\end{proof}

\begin{lemma}[Two-stage telescoping split]
	\label{lem:two_regimes_rates}
	\label{lem:two_stage_split}
	Let \(L_0,L_{2/3}\ge 0\), \(L_0+L_{2/3}>0\), let \(C\ge 0\), and let
	\(a_0,\dots,a_{T-1}\ge0\) satisfy
	\begin{equation}
		\label{eq:two_stage_basic_sum}
		\sum_{t=0}^{T-1}
		\frac{a_t^2}
		{L_0+2^{2/3}L_{2/3}a_t^{2/3}}
		\le C.
	\end{equation}
	Define the two sets of iterations
	\[
	\mathcal S_T
	:=
	\left\{
	t:
	2^{2/3}L_{2/3}a_t^{2/3}\le L_0
	\right\},
	\]
	\[
	\mathcal L_T
	:=
	\left\{
	t:
	2^{2/3}L_{2/3}a_t^{2/3}> L_0
	\right\}.
	\]
	Then
	\begin{equation}
		\label{eq:two_stage_small_sum}
		\sum_{t\in\mathcal S_T} a_t^2
		\le
		2L_0 C,
	\end{equation}
	and
	\begin{equation}
		\label{eq:two_stage_large_sum}
		\sum_{t\in\mathcal L_T} a_t^{4/3}
		\le
		2^{5/3}L_{2/3} C.
	\end{equation}
	Consequently, if \(N_S:=|\mathcal S_T|\) and \(N_L:=|\mathcal L_T|\), then,
	whenever \(N_S>0\),
	\begin{equation}
		\label{eq:two_stage_small_min}
		\min_{t\in\mathcal S_T} a_t
		\le
		\sqrt{\frac{2L_0 C}{N_S}},
	\end{equation}
	and, whenever \(N_L>0\),
	\begin{equation}
		\label{eq:two_stage_large_min}
		\min_{t\in\mathcal L_T} a_t
		\le
		\left(
		\frac{2^{5/3}L_{2/3} C}{N_L}
		\right)^{3/4}.
	\end{equation}
	If one of the sets is empty, the corresponding bound is simply ignored.
	If \(L_0=0\) or \(L_{2/3}=0\), the corresponding term is understood by
	continuity or omitted.
\end{lemma}

\begin{proof}
	Assume first that \(L_0>0\) and \(L_{2/3}>0\).
	
	For \(t\in\mathcal S_T\) we have
	\[
	2^{2/3}L_{2/3}a_t^{2/3}\le L_0.
	\]
	Hence
	\[
	L_0+2^{2/3}L_{2/3}a_t^{2/3}
	\le
	2L_0.
	\]
	Therefore
	\[
	\frac{a_t^2}
	{L_0+2^{2/3}L_{2/3}a_t^{2/3}}
	\ge
	\frac{a_t^2}{2L_0}.
	\]
	Summing over \(t\in\mathcal S_T\) and using
	\eqref{eq:two_stage_basic_sum}, we obtain
	\[
	\sum_{t\in\mathcal S_T}\frac{a_t^2}{2L_0}
	\le
	\sum_{t\in\mathcal S_T}
	\frac{a_t^2}
	{L_0+2^{2/3}L_{2/3}a_t^{2/3}}
	\le
	C.
	\]
	This proves \eqref{eq:two_stage_small_sum}.
	
	For \(t\in\mathcal L_T\) we have
	\[
	2^{2/3}L_{2/3}a_t^{2/3}>L_0.
	\]
	Hence
	\[
	L_0+2^{2/3}L_{2/3}a_t^{2/3}
	<
	2\cdot 2^{2/3}L_{2/3}a_t^{2/3}
	=
	2^{5/3}L_{2/3}a_t^{2/3}.
	\]
	Therefore
	\[
	\frac{a_t^2}
	{L_0+2^{2/3}L_{2/3}a_t^{2/3}}
	>
	\frac{a_t^2}
	{2^{5/3}L_{2/3}a_t^{2/3}}
	=
	\frac{a_t^{4/3}}
	{2^{5/3}L_{2/3}}.
	\]
	Summing over \(t\in\mathcal L_T\) and using
	\eqref{eq:two_stage_basic_sum}, we obtain
	\[
	\sum_{t\in\mathcal L_T}
	\frac{a_t^{4/3}}
	{2^{5/3}L_{2/3}}
	\le
	\sum_{t\in\mathcal L_T}
	\frac{a_t^2}
	{L_0+2^{2/3}L_{2/3}a_t^{2/3}}
	\le
	C.
	\]
	This proves \eqref{eq:two_stage_large_sum}.
	
	The bounds \eqref{eq:two_stage_small_min} and
	\eqref{eq:two_stage_large_min} follow by taking the minimum over the
	corresponding finite sets. Indeed, if \(N_S>0\), then
	\[
	\sum_{t\in\mathcal S_T}a_t^2
	\ge
	N_S\min_{t\in\mathcal S_T}a_t^2,
	\]
	and similarly for \(\mathcal L_T\).
	
	The degenerate cases \(L_0=0\) or \(L_{2/3}=0\) follow by the same
	argument with the corresponding regime omitted, or by a limiting argument.
\end{proof}

\begin{theorem}[Two types of iterations for backtracking]
	\label{cor:backtracking_two_regimes}
	For Algorithm~\ref{alg:adaptive_gd_backtrack}, fix \(T\ge1\), and let
	\[
	\Delta = f(X^0)-f^*,
	\qquad
	a_t:=\norm{\nabla f(X^t)}_F,
	\qquad
	t=0,\dots,T-1.
	\]
	Define the two sets of iterations
	\[
	\mathcal S_T
	:=
	\left\{
	t:
	2^{2/3}L_{2/3}a_t^{2/3}\le L_0
	\right\},
	\]
	\[
	\mathcal L_T
	:=
	\left\{
	t:
	2^{2/3}L_{2/3}a_t^{2/3}> L_0
	\right\},
	\]
	and let
	\[
	N_S:=|\mathcal S_T|,
	\qquad
	N_L:=|\mathcal L_T|.
	\]
	Then
	\begin{equation}\label{eq:backtracking_small_sum}
		\sum_{t\in\mathcal S_T} a_t^2
		\le
		\frac{2L_0\Delta}{c},
	\end{equation}
	and
	\begin{equation}\label{eq:backtracking_large_sum}
		\sum_{t\in\mathcal L_T} a_t^{4/3}
		\le
		\frac{2^{5/3}L_{2/3}\Delta}{c}.
	\end{equation}
	Consequently, if \(N_S>0\), then
	\begin{equation}\label{eq:backtracking_small_min}
		\min_{t\in\mathcal S_T}
		\norm{\nabla f(X^t)}_F
		\le
		\sqrt{
			\frac{2L_0\Delta}{cN_S}
		},
	\end{equation}
	and if \(N_L>0\), then
	\begin{equation}\label{eq:backtracking_large_min}
		\min_{t\in\mathcal L_T}
		\norm{\nabla f(X^t)}_F
		\le
		\left(
		\frac{2^{5/3}L_{2/3}\Delta}{cN_L}
		\right)^{3/4}.
	\end{equation}
	
	In particular, for any \(\varepsilon>0\), in order to guarantee
	\[
	\min_{0\le t<T}\norm{\nabla f(X^t)}_F\le\varepsilon,
	\]
	it suffices to take
	\begin{equation}\label{eq:backtracking_additive_complexity}
		T
		\ge
		\frac{2L_0\Delta}{c\varepsilon^2}
		+
		\frac{2^{5/3}L_{2/3}\Delta}{c\varepsilon^{4/3}}.
	\end{equation}
	Equivalently, one may write the explicit no-max rate
	\begin{equation}\label{eq:backtracking_no_max_rate}
		\min_{0\le t<T}\norm{\nabla f(X^t)}_F
		\le
		\sqrt{
			\frac{4L_0\Delta}{cT}
		}
		+
		\left(
		\frac{2^{8/3}L_{2/3}\Delta}{cT}
		\right)^{3/4}.
	\end{equation}
\end{theorem}

\begin{proof}
	By Theorem~\ref{teo:backtrack_convergence},
	\[
	c\sum_{t=0}^{T-1}
	\frac{a_t^2}{\ell(2a_t)}
	\le
	\Delta.
	\]
	Hence
	\[
	\sum_{t=0}^{T-1}
	\frac{a_t^2}{\ell(2a_t)}
	\le
	\frac{\Delta}{c}.
	\]
	Applying Lemma~\ref{lem:two_stage_split} with
	\[
	C=\frac{\Delta}{c}
	\]
	gives \eqref{eq:backtracking_small_sum} and
	\eqref{eq:backtracking_large_sum}.
	The bounds \eqref{eq:backtracking_small_min} and
	\eqref{eq:backtracking_large_min} follow from
	Lemma~\ref{lem:two_stage_split}.
	
	We now prove the complexity statement. Suppose that
	\[
	\min_{0\le t<T}a_t>\varepsilon.
	\]
	Then for every \(t\in\mathcal S_T\) we have \(a_t^2>\varepsilon^2\), and hence
	\[
	N_S\varepsilon^2
	<
	\sum_{t\in\mathcal S_T}a_t^2
	\le
	\frac{2L_0\Delta}{c}.
	\]
	Therefore
	\[
	N_S
	<
	\frac{2L_0\Delta}{c\varepsilon^2}.
	\]
	Similarly, for every \(t\in\mathcal L_T\) we have
	\(a_t^{4/3}>\varepsilon^{4/3}\), and hence
	\[
	N_L\varepsilon^{4/3}
	<
	\sum_{t\in\mathcal L_T}a_t^{4/3}
	\le
	\frac{2^{5/3}L_{2/3}\Delta}{c}.
	\]
	Therefore
	\[
	N_L
	<
	\frac{2^{5/3}L_{2/3}\Delta}{c\varepsilon^{4/3}}.
	\]
	Since \(T=N_S+N_L\), we obtain
	\[
	T
	<
	\frac{2L_0\Delta}{c\varepsilon^2}
	+
	\frac{2^{5/3}L_{2/3}\Delta}{c\varepsilon^{4/3}},
	\]
	which contradicts \eqref{eq:backtracking_additive_complexity}.
	Thus
	\[
	\min_{0\le t<T}a_t\le\varepsilon.
	\]

	Finally, to obtain the explicit no-max bound
	\eqref{eq:backtracking_no_max_rate}, set
	\[
	\varepsilon
	:=
	\sqrt{
		\frac{4L_0\Delta}{cT}
	}
	+
	\left(
	\frac{2^{8/3}L_{2/3}\Delta}{cT}
	\right)^{3/4}.
	\]
	Since
	\[
	\varepsilon
	\ge
	\sqrt{
		\frac{4L_0\Delta}{cT}
	},
	\]
	we have
	\[
	\frac{2L_0\Delta}{c\varepsilon^2}
	\le
	\frac{T}{2}.
	\]
	Similarly, since
	\[
	\varepsilon
	\ge
	\left(
	\frac{2^{8/3}L_{2/3}\Delta}{cT}
	\right)^{3/4},
	\]
	we have
	\[
	\frac{2^{5/3}L_{2/3}\Delta}{c\varepsilon^{4/3}}
	\le
	\frac{T}{2}.
	\]
	Hence
	\[
	\frac{2L_0\Delta}{c\varepsilon^2}
	+
	\frac{2^{5/3}L_{2/3}\Delta}{c\varepsilon^{4/3}}
	\le
	T.
	\]
	By the already proved complexity statement,
	\[
	\min_{0\le t<T}a_t\le\varepsilon,
	\]
	which proves \eqref{eq:backtracking_no_max_rate}.
\end{proof}

\begin{remark}[Interpretation of the two types of iterations]
	Corollary~\ref{cor:backtracking_two_regimes} separates the iterations into
	two types: small-gradient iterations and large-gradient iterations.
	The sets \(\mathcal S_T\) and \(\mathcal L_T\) are not assumed to be
	consecutive time intervals; they may interleave arbitrarily.
	The proof uses only the monotone decrease of the objective function and the
	quantitative descent estimate. No monotonicity of the gradient norm,
	quasiconvexity, or gradient dominance condition is required.
	
	The additive complexity \eqref{eq:backtracking_additive_complexity} is an
	unconditional no-max guarantee. If the numbers \(N_S\) and \(N_L\) are known
	a posteriori, the sharper regime-specific bounds
	\eqref{eq:backtracking_small_min} and
	\eqref{eq:backtracking_large_min} should be used.
\end{remark}

\section{Unconstrained SafeAdam-\(\ell\)}

We now turn to the analysis of the unconstrained SafeAdam-\(\ell\) method
(Algorithm~\ref{alg:safeadam_unconstrained}) for solving the low-rank matrix
factorization problem
\[
\min_{X \in \mathbb{R}^{n \times r}} f(X) := F(XX^\top),
\]
where \(F\) satisfies Assumption~\ref{asm:quadratic}. As established in
Section~\ref{sec:generalized_smoothness}, the factorized function \(f\)
possesses the generalized \((2/3,L_0,L_{2/3})\)-smoothness property. The method
combines the Adam direction with a descent-angle certificate and a fallback to
the normalized gradient step, and it chooses the step size using the
\(q\)-function of generalized smoothness. Our goal is to establish the global
convergence rate of this method in terms of the gradient norm. To handle the
case of a zero gradient, we adopt the convention of
Definition~\ref{def:zero_gradient_convention}. The following theorem provides
the main convergence guarantee together with explicit complexity bounds.

\begin{theorem}[Convergence of SafeAdam-$\ell$]
	\label{teo:safe_momentum_strong_v2}
	Let \(f\) be bounded from below:
	\[
	f(X)\ge f^*,
	\]
	and satisfy the \((2/3,L_0,L_{2/3})\)-smoothness, where
	\[
	L_0\ge0,\qquad L_{2/3}\ge0,
	\]
	and moreover
	\[
	L_0>0 \quad \text{or} \quad L_{2/3}>0.
	\]

	If for some \(\tau<T\) one has
	\[
	\nabla f(X^\tau)=0,
	\]
	then the algorithm stops at iteration \(\tau\).
	By Definition~\ref{def:zero_gradient_convention},
	we extend the sequence by setting
	\[
	X^t:=X^\tau
	\qquad
	\forall\,t\ge \tau.
	\]
	All summands with \(t\ge\tau\) are zero.
	The finite sum over the iterations before \(\tau\) is bounded by the same
	telescoping argument applied up to the stopping time, and the minimum is zero.
	Hence it suffices to give the proof for the case where
	\[
	\|\nabla f(X^t)\|_F>0
	\qquad
	\forall\,t=0,\dots,T-1.
	\]

	Otherwise, assume that
	\[
	\|\nabla f(X^t)\|_F>0
	\qquad
	\forall\,t=0,\dots,T-1.
	\]
	
	Then, for the sequence generated by the unconstrained SafeAdam-\(\ell\) algorithm,
	for any \(T\ge1\) it holds that
	\[
	\sum_{t=0}^{T-1}
	\frac{\norm{\nabla f(X^t)}_F^2}
	{L_0+2^{2/3}L_{2/3}\norm{\nabla f(X^t)}_F^{2/3}}
	\le
	2\Delta.
	\]
	Consequently,
	\[
	\min_{0\le t<T}
	\frac{\norm{\nabla f(X^t)}_F^2}
	{L_0+2^{2/3}L_{2/3}\norm{\nabla f(X^t)}_F^{2/3}}
	\le
	\frac{2\Delta}{T},
	\]
	where
	\[
	\Delta = f(X^0) - f^*.
	\]
\end{theorem}

\begin{proof}
	At every iteration the algorithm chooses a direction \(d_t\) and a step size such that 
	the function value is at most that of the projected gradient step. Hence,
	\[
	f(X^{t+1}) \le f(X_t^{\rm PG}),
	\]
	where \(X_t^{\rm PG}=X^t-\gamma_t^{\rm PG}g_t\) is the unconstrained gradient step
	with \(\gamma_t^{\rm PG}=q(a_t;a_t)/a_t\).
	
	By Lemma~\ref{lem:descent_direction} applied to the gradient step, we have
	\[
	f(X_t^{\rm PG})
	\le f(X^t)
	-
	\int_0^{a_t}
	\frac{a_t-s}{\ell(a_t+s)}\,ds.
	\]
	We now use the stronger lower bound
	\[
	\int_0^{a_t}
	\frac{a_t-s}{\ell(a_t+s)}\,ds
	\ge
	\frac{a_t^2}{2\,\ell(2a_t)}.
	\]
	This follows because for \(s\in[0,a_t]\), \(\ell(a_t+s)\le \ell(2a_t)\), so
	\[
	\int_0^{a_t} \frac{a_t-s}{\ell(a_t+s)}\,ds
	\ge
	\frac{1}{\ell(2a_t)}\int_0^{a_t}(a_t-s)\,ds
	=
	\frac{a_t^2}{2\ell(2a_t)}.
	\]
	Thus,
	\[
	f(X^{t+1})
	\le f(X^t)
	-
	\frac{a_t^2}{2\,\ell(2a_t)}.
	\]
	
	Summing over \(t=0,\dots,T-1\), we obtain
	\[
	\frac12 \sum_{t=0}^{T-1} \frac{a_t^2}{\ell(2a_t)}
	\le f(X^0)-f(X^T) \le \Delta.
	\]
	Hence,
	\[
	\sum_{t=0}^{T-1} \frac{a_t^2}{\ell(2a_t)}
	\le 2\Delta.
	\]
	The minimum does not exceed the average, so
	\[
	\min_{0\le t<T} \frac{a_t^2}{\ell(2a_t)}
	\le
	\frac{2\Delta}{T}.
	\]
	Since \(\ell(2a_t)=L_0+2^{2/3}L_{2/3}\,a_t^{2/3}\), the claim follows.
\end{proof}

\begin{remark}[Comparison with the theoretical step]
	For the basic adaptive gradient descent
	(Algorithm~\ref{alg:adaptive_gd_tyurin}) with step $\eta_t^{\mathrm{theory}} = \frac{1}{a_t} q(a_t; a_t)$,
	the same descent analysis as in Theorem~\ref{teo:safe_momentum_strong_v2}
	yields the bound
	\[
	\min_{0\le t<T} \frac{a_t^2}{\ell(2a_t)} \le \frac{2\Delta}{T}.
	\]
	For the backtracking method, Theorem~\ref{teo:backtrack_convergence} gives
	\[
	\min_{0\le t<T} \frac{a_t^2}{\ell(2a_t)} \le \frac{\Delta}{cT}.
	\]
	For $c = 0{.}1$ the formal constant is worse by a factor of
	$\frac{1}{2c} = 5$.
	However, in computational experiments the increased step $\alpha_t \cdot \frac{1}{a_t} q(a_t; a_t)$
	with $\alpha_t$ up to 5 leads to a substantially larger
	real descent at every iteration,
	which fully compensates and outweighs
	the deterioration of the theoretical constant.
\end{remark}

\begin{remark}[Computational cost]
	At every iteration Algorithm~\ref{alg:adaptive_gd_backtrack}
	requires:
	\begin{itemize}
		\item One evaluation of the gradient $\nabla f(X^t)$;
		\item At most $\lceil\log_2 \alpha_{\max}\rceil + 1 = 4$
		evaluations of the function $f$ (for checking the Armijo condition);
		\item One evaluation of the integral $q(a_t; a_t)$.
	\end{itemize}
	For the matrix factorization problem
	$f(X) = \frac{1}{2}\|M - XX^\top\|_F^2$ with a sparse $M$
	the evaluation of the function costs $O(nr^2 + \mathrm{nnz}(M) \cdot r)$,
	which is 1.5--2 times cheaper than the evaluation of the gradient.
	Therefore the additional costs of backtracking
	do not exceed 50\% of the cost of one iteration,
	which is compensated many times over by the reduction
	of the total number of iterations.
\end{remark}

\begin{corollary}[Two types of iterations for SafeAdam-\(\ell\)]
	\label{cor:rate_safeadam_v2}
	\label{cor:safeadam_two_regimes}
	\label{cor:complexity_safeadam_v2}
	For the unconstrained SafeAdam-\(\ell\) algorithm, fix \(T\ge1\), and let
	\[
	\Delta = f(X^0)-f^*,
	\qquad
	a_t:=\norm{\nabla f(X^t)}_F,
	\qquad
	t=0,\dots,T-1.
	\]
	Define the two sets of iterations
	\[
	\mathcal S_T
	:=
	\left\{
	t:
	2^{2/3}L_{2/3}a_t^{2/3}\le L_0
	\right\},
	\]
	\[
	\mathcal L_T
	:=
	\left\{
	t:
	2^{2/3}L_{2/3}a_t^{2/3}> L_0
	\right\},
	\]
	and let
	\[
	N_S:=|\mathcal S_T|,
	\qquad
	N_L:=|\mathcal L_T|.
	\]
	Then
	\begin{equation}\label{eq:safeadam_small_sum}
		\sum_{t\in\mathcal S_T} a_t^2
		\le
		4L_0\Delta,
	\end{equation}
	and
	\begin{equation}\label{eq:safeadam_large_sum}
		\sum_{t\in\mathcal L_T} a_t^{4/3}
		\le
		2^{8/3}L_{2/3}\Delta.
	\end{equation}
	Consequently, if \(N_S>0\), then
	\begin{equation}\label{eq:safeadam_small_min}
		\min_{t\in\mathcal S_T}
		\norm{\nabla f(X^t)}_F
		\le
		\sqrt{
			\frac{4L_0\Delta}{N_S}
		},
	\end{equation}
	and if \(N_L>0\), then
	\begin{equation}\label{eq:safeadam_large_min}
		\min_{t\in\mathcal L_T}
		\norm{\nabla f(X^t)}_F
		\le
		\left(
		\frac{2^{8/3}L_{2/3}\Delta}{N_L}
		\right)^{3/4}.
	\end{equation}
	
	In particular, for any \(\varepsilon>0\), in order to guarantee
	\[
	\min_{0\le t<T}\norm{\nabla f(X^t)}_F\le\varepsilon,
	\]
	it suffices to take
	\begin{equation}\label{eq:safeadam_additive_complexity}
		T
		\ge
		\frac{4L_0\Delta}{\varepsilon^2}
		+
		\frac{2^{8/3}L_{2/3}\Delta}{\varepsilon^{4/3}}.
	\end{equation}
	Equivalently, one may write the explicit no-max rate
	\begin{equation}\label{eq:safeadam_no_max_rate}
		\min_{0\le t<T}\norm{\nabla f(X^t)}_F
		\le
		\sqrt{
			\frac{8L_0\Delta}{T}
		}
		+
		\left(
		\frac{2^{11/3}L_{2/3}\Delta}{T}
		\right)^{3/4}.
	\end{equation}
\end{corollary}

\begin{proof}
	The proof of Theorem~\ref{teo:safe_momentum_strong_v2} gives
	\[
	\sum_{t=0}^{T-1}
	\frac{a_t^2}{\ell(2a_t)}
	\le
	2\Delta,
	\]
	where
	\[
	\ell(2a_t)
	=
	L_0+2^{2/3}L_{2/3}a_t^{2/3}.
	\]
	Applying Lemma~\ref{lem:two_stage_split} with
	\[
	C=2\Delta
	\]
	yields \eqref{eq:safeadam_small_sum} and
	\eqref{eq:safeadam_large_sum}.
	The bounds \eqref{eq:safeadam_small_min} and
	\eqref{eq:safeadam_large_min} follow from
	Lemma~\ref{lem:two_stage_split}.
	
	We now prove the complexity statement. Suppose that
	\[
	\min_{0\le t<T}a_t>\varepsilon.
	\]
	Then for every \(t\in\mathcal S_T\) we have \(a_t^2>\varepsilon^2\), and hence
	\[
	N_S\varepsilon^2
	<
	\sum_{t\in\mathcal S_T}a_t^2
	\le
	4L_0\Delta.
	\]
	Therefore
	\[
	N_S
	<
	\frac{4L_0\Delta}{\varepsilon^2}.
	\]
	Similarly, for every \(t\in\mathcal L_T\) we have
	\(a_t^{4/3}>\varepsilon^{4/3}\), and hence
	\[
	N_L\varepsilon^{4/3}
	<
	\sum_{t\in\mathcal L_T}a_t^{4/3}
	\le
	2^{8/3}L_{2/3}\Delta.
	\]
	Therefore
	\[
	N_L
	<
	\frac{2^{8/3}L_{2/3}\Delta}{\varepsilon^{4/3}}.
	\]
	Since \(T=N_S+N_L\), we obtain
	\[
	T
	<
	\frac{4L_0\Delta}{\varepsilon^2}
	+
	\frac{2^{8/3}L_{2/3}\Delta}{\varepsilon^{4/3}},
	\]
	which contradicts \eqref{eq:safeadam_additive_complexity}.
	Thus
	\[
	\min_{0\le t<T}a_t\le\varepsilon.
	\]
	
	Finally, to obtain the explicit no-max bound
	\eqref{eq:safeadam_no_max_rate}, set
	\[
	\varepsilon
	:=
	\sqrt{
		\frac{8L_0\Delta}{T}
	}
	+
	\left(
	\frac{2^{11/3}L_{2/3}\Delta}{T}
	\right)^{3/4}.
	\]
	Then
	\[
	\frac{4L_0\Delta}{\varepsilon^2}
	\le
	\frac{T}{2},
	\]
	and
	\[
	\frac{2^{8/3}L_{2/3}\Delta}{\varepsilon^{4/3}}
	\le
	\frac{T}{2}.
	\]
	Hence
	\[
	\frac{4L_0\Delta}{\varepsilon^2}
	+
	\frac{2^{8/3}L_{2/3}\Delta}{\varepsilon^{4/3}}
	\le
	T.
	\]
	By the already proved complexity statement,
	\[
	\min_{0\le t<T}a_t\le\varepsilon,
	\]
	which proves \eqref{eq:safeadam_no_max_rate}.
\end{proof}

\begin{remark}[Interpretation of the two types of iterations for SafeAdam]
	Corollary~\ref{cor:safeadam_two_regimes} separates the iterations of the
	unconstrained SafeAdam-\(\ell\) algorithm into two types: small-gradient
	iterations and large-gradient iterations.
	The sets \(\mathcal S_T\) and \(\mathcal L_T\) may interleave arbitrarily.
	The proof uses only the monotone decrease of the objective function and the
	descent certificate of SafeAdam-\(\ell\). No monotonicity of the gradient
	norm, quasiconvexity, or gradient dominance condition is required.
	
	The additive complexity \eqref{eq:safeadam_additive_complexity} is an
	unconditional no-max guarantee. If the numbers \(N_S\) and \(N_L\) are known
	a posteriori, the sharper regime-specific bounds
	\eqref{eq:safeadam_small_min} and
	\eqref{eq:safeadam_large_min} should be used.
\end{remark}

\begin{lemma}[Lower bound for \(b\ge ca\)]
	Let \(g=\nabla f(X)\), \(a=\norm{g}_F\), and let \(d\) be a unit direction with
	\[
	b=\inprod{g}{d}\ge0.
	\]
	If \(a>0\) and \(b\ge c a\) for some \(c>0\), then
	\[
	\int_0^b \frac{b-s}{\ell(a+s)}\,ds
	\;\ge\;
	\frac{c^2a^2}{4\ell(2a)}.
	\]
	If \(a=0\), the inequality holds trivially because both sides equal zero under Convention~\ref{def:zero_gradient_convention}.
\end{lemma}

\begin{proof}
	If \(a=0\), then \(g=0\), so \(b=\inprod{0}{d}=0\). The integral is \(0\), and the right-hand side is \(0\) by convention; the statement is true.
	Assume now \(a>0\).
	Since \(\norm{d}=1\), by Cauchy--Schwarz, \(b\le a\).
	Restrict the integral to \([0,b/2]\):
	\[
	\int_0^b \frac{b-s}{\ell(a+s)}\,ds
	\;\ge\;
	\int_0^{b/2} \frac{b-s}{\ell(a+s)}\,ds.
	\]
	For \(s\in[0,b/2]\), we have \(b-s\ge b/2\), and \(a+s\le a+b/2\le 2a\). Since \(\ell\) is nondecreasing, \(\ell(a+s)\le \ell(2a)\). Hence
	\[
	\int_0^{b/2} \frac{b-s}{\ell(a+s)}\,ds
	\;\ge\;
	\int_0^{b/2} \frac{b/2}{\ell(2a)}\,ds
	=
	\frac{b^2}{4\ell(2a)}.
	\]
	Because \(b\ge c a\), we obtain
	\[
	\frac{b^2}{4\ell(2a)}
	\;\ge\;
	\frac{c^2 a^2}{4\ell(2a)}.
	\]
	This proves the lemma.
\end{proof}

\section{Projected Certified Adam-type Method}

Let us pass to the problem
\[
\min_{X\ge0} f(X).
\]

Let us define the gradient mapping:
\[
G_\gamma(X)
=
\frac{1}{\gamma}
\left(
X-\mathrm{Proj}(X-\gamma\nabla f(X))
\right).
\]

\begin{lemma}[Projected descent for a gradient step]
	\label{lem:projected_descent}
	Let \(X\ge0\), \(g=\nabla f(X)\), \(a=\norm{g}_F\), and let
	\[
	\gamma=
	\begin{cases}
		\dfrac{q(a;a)}{a}, & a>0,\\[6pt]
		0, & a=0.
	\end{cases}
	\]
	Set
	\[
	X^+=\mathrm{Proj}(X-\gamma g),
	\]
	and define
	\[
	G=
	\begin{cases}
		\dfrac{X-X^+}{\gamma}, & \gamma>0,\\[6pt]
		0, & \gamma=0.
	\end{cases}
	\]
	Then
	\[
	f(X^+)
	\le
	f(X)
	-
	\frac{\gamma}{2}\norm{G}_F^2.
	\]
\end{lemma}

\begin{proof}
	If \(a=0\), then \(\gamma=0\), \(X^+=X\), \(G=0\), and the inequality is trivial.
	Further we assume \(a>0\).
	
	From the property of the projection onto a convex set, for any \(Y\in Q=\mathbb{R}^{n\times r}_{\ge0}\) we have
	\[
	\inprod{X-\gamma g-X^+}{Y-X^+}\le0.
	\]
	We take \(Y=X\). Then
	\[
	\inprod{X-\gamma g-X^+}{X-X^+}\le0.
	\]
	Since
	\[
	X-X^+=\gamma G,
	\]
	we obtain
	\[
	\inprod{\gamma G-\gamma g}{\gamma G} \le 0,
	\]
	that is,
	\[
	\inprod{g}{G}\ge\norm{G}_F^2.
	\]
	
	Denote
	\[
	r = \norm{G}_F.
	\]
	
	From the nonexpansiveness of the projection,
	\[
	\norm{X^+-X}_F
	\le
	\norm{(X-\gamma g)-X}_F
	=
	\gamma a.
	\]
	Consequently,
	\[
	r = \frac{\norm{X^+-X}_F}{\gamma}\le a.
	\]
	If \(r = 0\), then \(X^+=X\), \(G=0\), and the inequality
	\(f(X^+) \le f(X) - \frac{\gamma}{2}\norm{G}_F^2\) is trivial.
	Therefore further we assume \(r > 0\).
	
	By the integral lemma on the function increment,
	\[
	f(X^+)
	\le
	f(X)+\inprod{g}{X^+-X}
	+
	\int_0^{\norm{X^+-X}_F}
	q^{-1}(\tau;a)\,d\tau.
	\]
	Since
	\[
	X^+-X=-\gamma G,
	\]
	we have
	\[
	\inprod{g}{X^+-X}
	=
	-\gamma\inprod{g}{G}
	\le
	-\gamma r^2.
	\]
	Therefore
	\[
	f(X^+)
	\le
	f(X)-\gamma r^2
	+
	\int_0^{\gamma r}q^{-1}(\tau;a)\,d\tau.
	\]
	
	Let us estimate the last integral. Since the function
	\[
	v\mapsto \frac{1}{\ell(a+v)}
	\]
	is nonincreasing on \([0,\infty)\), the average over \([0,r]\) is not smaller
	than the average over \([0,a]\) whenever \(0<r\le a\). Hence
	\[
	\frac{q(r;a)}{r}
	=
	\frac{1}{r}\int_0^r\frac{dv}{\ell(a+v)}
	\ge
	\frac{1}{a}\int_0^a\frac{dv}{\ell(a+v)}
	=
	\frac{q(a;a)}{a}
	=
	\gamma.
	\]
	Multiplying by \(r\), we obtain
	\[
	q(r;a)\ge \gamma r.
	\]
	Since \(q^{-1}(\cdot;a)\) is increasing, this implies
	\[
	q^{-1}(\gamma r;a)\le r.
	\]
	
	The function \(q^{-1}(\cdot;a)\) is convex. To see this, note first that \(q(\cdot;a)\) is concave on \([0,\infty)\) because its derivative
	\[
	q'(s;a) = \frac{1}{\ell(a+s)}
	\]
	is nonincreasing (since \(\ell\) is nondecreasing). A standard result states that the inverse of a strictly increasing concave function is convex. For completeness, we include the short proof.
	
	Let \(F = q(\cdot;a)\). Then \(F\) is strictly increasing and concave. Take any \(0 \le u < v\) and \(\lambda \in [0,1]\). By concavity of \(F\),
	\[
	F\bigl(\lambda F^{-1}(u) + (1-\lambda) F^{-1}(v)\bigr)
	\ge
	\lambda F(F^{-1}(u)) + (1-\lambda) F(F^{-1}(v))
	= \lambda u + (1-\lambda) v.
	\]
	Since \(F^{-1}\) is increasing, applying it to both sides yields
	\[
	\lambda F^{-1}(u) + (1-\lambda) F^{-1}(v)
	\ge
	F^{-1}\bigl(\lambda u + (1-\lambda) v\bigr).
	\]
	Thus \(F^{-1} = q^{-1}(\cdot;a)\) is convex.
	
	For a convex function \(\phi\) with \(\phi(0) = 0\) it holds that
	\[
	\int_0^z\phi(\tau)\,d\tau
	\le
	\frac{z\phi(z)}{2}.
	\]
	
	Applying this to \(\phi = q^{-1}(\cdot; a)\) and \(z = \gamma r\), we obtain
	\[
	\int_0^{\gamma r}q^{-1}(\tau;a)\,d\tau
	\le
	\frac{\gamma r}{2}q^{-1}(\gamma r;a)
	\le
	\frac{\gamma r^2}{2}.
	\]
	
	Consequently,
	\[
	f(X^+)
	\le
	f(X)-\gamma r^2+\frac{\gamma r^2}{2}
	=
	f(X)-\frac{\gamma}{2}\norm{G}_F^2.
	\]
	The lemma is proved.
\end{proof}

\begin{theorem}[Convergence of the certified projected method]
	\label{teo:projected_convergence}
	Let \(f\) be bounded from below and satisfy the \((2/3,L_0,L_{2/3})\)-smoothness,
	where
	\[
	L_0\ge0,\qquad L_{2/3}\ge0,
	\]
	and let the condition
	\[
	L_0>0 \quad \text{or} \quad L_{2/3}>0
	\]
	hold.
	
	If at some iteration \(\tau<T\) it holds that \(\nabla f(X^\tau)=0\),
	then the algorithm stops.
	By Definition~\ref{def:zero_gradient_convention}, the sequence is extended
	by stationarity.
	All summands with \(t\ge\tau\) are zero.
	The finite sum over the iterations before \(\tau\) is bounded by the same
	telescoping argument applied up to the stopping time, and the minimum is zero.
	
	Otherwise, for the sequence generated by the certified projected SafeAdam-\(\ell\), 
	for any \(T\ge1\) it holds that
	\[
	\sum_{t=0}^{T-1}
	\frac{\norm{G_t}_F^2}
	{L_0+2^{2/3}L_{2/3}\norm{\nabla f(X^t)}_F^{2/3}}
	\le
	2\Delta.
	\]
	Consequently,
	\[
	\min_{0\le t<T}
	\frac{\norm{G_t}_F^2}
	{L_0+2^{2/3}L_{2/3}\norm{\nabla f(X^t)}_F^{2/3}}
	\le
	\frac{2\Delta}{T},
	\]
	where
	\[
	\Delta=f(X^0)-f^*.
	\]
\end{theorem}

\begin{proof}
	By the construction of the algorithm
	\[
	f(X^{t+1})\le f(X_t^{\rm PG}).
	\]
	By the lemma on the projected descent,
	\[
	f(X_t^{\rm PG})
	\le
	f(X^t)
	-
	\frac{\gamma_t^{\rm PG}}{2}\norm{G_t}^2.
	\]
	Consequently,
	\[
	f(X^{t+1})
	\le
	f(X^t)
	-
	\frac{\gamma_t^{\rm PG}}{2}\norm{G_t}^2.
	\]
	
	Since \(\ell\) is nondecreasing,
	\[
	\gamma_t^{\rm PG}
	=
	\int_0^1\frac{dv}{\ell(a_t+a_tv)}
	\ge
	\frac{1}{\ell(2a_t)}.
	\]
	Therefore
	\[
	f(X^{t+1})
	\le
	f(X^t)
	-
	\frac{\norm{G_t}^2}{2\ell(2a_t)}.
	\]
	Summing over \(t=0,\dots,T-1\), we obtain
	\[
	\sum_{t=0}^{T-1}
	\frac{\norm{G_t}^2}{\ell(2a_t)}
	\le
	2\Delta.
	\]
	The minimum does not exceed the average:
	\[
	\min_{0\le t<T}
	\frac{\norm{G_t}^2}{\ell(2a_t)}
	\le
	\frac{2\Delta}{T}.
	\]
	The theorem is proved.
\end{proof}

\begin{lemma}[Boundedness of the Initial Level Set and of the Gradients]
	\label{lem:bounded_gradients}
	Let
	\[
	F(Y)=\frac12\langle\mathcal A Y,Y\rangle+\langle B,Y\rangle
	\]
	satisfy Assumption~\ref{asm:quadratic}, and let \(f(X)=F(XX^\top)\).
	Let \(X^0\ge0\) be the starting point of Algorithm~\ref{alg:safeadam_projected}, and define
	\[
	C_0:=\max\{f(X^0),0\}.
	\]
	Then the feasible initial level set
	\[
	\mathcal L_0
	:=
	\left\{
	X\in\mathbb R^{n\times r}_{\ge0}
	\;:\;
	f(X)\le f(X^0)
	\right\}
	\]
	is bounded. More explicitly, every \(X\in\mathcal L_0\) satisfies
	\[
	\norm{X}_F^2\le R^2,
	\]
	where
	\[
	R^2
	:=
	\frac{r\norm{B}_F}{\mu}
	+
	\sqrt{
		\left(\frac{r\norm{B}_F}{\mu}\right)^2
		+
		\frac{2r C_0}{\mu}
	}.
	\]
	Consequently, for every \(X\in\mathcal L_0\),
	\[
	\norm{\nabla f(X)}_F
	\le
	G_{\text{max}},
	\]
	where
	\[
	G_{\text{max}}
	:=
	2\left(
	\norm{\mathcal A}_{\mathbb S,\mathrm{op}} R^2
	+
	\norm{B}_F
	\right)R.
	\]
	In particular, for the sequence generated by Algorithm~\ref{alg:safeadam_projected},
	\[
	\norm{\nabla f(X^t)}_F\le G_{\text{max}},
	\qquad
	t=0,1,2,\dots
	\]
\end{lemma}

\begin{proof}
	Let \(X\in\mathbb R^{n\times r}\) and set \(S=XX^\top\in\mathbb S^{n\times n}\).
	By the \(\mu\)-positive definiteness of \(\mathcal A\) on symmetric matrices,
	\[
	\langle\mathcal A S,S\rangle
	\ge
	\mu\norm{S}_F^2.
	\]
	Also, by the Cauchy--Schwarz inequality,
	\[
	\langle B,S\rangle
	\ge
	-\norm{B}_F\norm{S}_F.
	\]
	Therefore
	\[
	f(X)
	=
	\frac12\langle\mathcal A S,S\rangle+\langle B,S\rangle
	\ge
	\frac{\mu}{2}\norm{S}_F^2
	-
	\norm{B}_F\norm{S}_F.
	\]
	
	If \(\sigma_1,\dots,\sigma_r\) are the singular values of \(X\), then
	\[
	\norm{S}_F^2
	=
	\norm{XX^\top}_F^2
	=
	\sum_{i=1}^r\sigma_i^4
	\ge
	\frac1r\left(\sum_{i=1}^r\sigma_i^2\right)^2
	=
	\frac1r\norm{X}_F^4.
	\]
	Hence
	\[
	\norm{S}_F
	\ge
	\frac{1}{\sqrt r}\norm{X}_F^2.
	\]
	By Lemma~\ref{lem:norm_XX},
	\[
	\norm{S}_F
	=
	\norm{XX^\top}_F
	\le
	\norm{X}_F^2.
	\]
	Combining these estimates gives
	\[
	f(X)
	\ge
	\frac{\mu}{2r}\norm{X}_F^4
	-
	\norm{B}_F\norm{X}_F^2.
	\]
	
	Now let \(X\in\mathcal L_0\). Since \(f(X)\le f(X^0)\le C_0\), with
	\[
	z:=\norm{X}_F^2,
	\]
	we obtain
	\[
	\frac{\mu}{2r}z^2-\norm{B}_F z
	\le
	C_0.
	\]
	Thus \(z\ge0\) cannot exceed the positive root of the quadratic equation
	\[
	\frac{\mu}{2r}z^2-\norm{B}_F z-C_0=0.
	\]
	Therefore
	\[
	\norm{X}_F^2=z
	\le
	\frac{r\norm{B}_F}{\mu}
	+
	\sqrt{
		\left(\frac{r\norm{B}_F}{\mu}\right)^2
		+
		\frac{2r C_0}{\mu}
	}
	=:R^2.
	\]
	This proves that \(\mathcal L_0\) is bounded.
	
	Next, using the gradient formula
	\[
	\nabla f(X)=2\bigl(\mathcal A(XX^\top)+B\bigr)X,
	\]
	we get
	\[
	\begin{aligned}
		\norm{\nabla f(X)}_F
		&\le
		2\norm{\mathcal A(XX^\top)X}_F
		+
		2\norm{BX}_F
		\\
		&\le
		2\norm{\mathcal A(XX^\top)}_F\norm{X}_F
		+
		2\norm{B}_F\norm{X}_F
		\\
		&\le
		2\left(
		\norm{\mathcal A}_{\mathbb S,\mathrm{op}}\norm{XX^\top}_F
		+
		\norm{B}_F
		\right)\norm{X}_F
		\\
		&\le
		2\left(
		\norm{\mathcal A}_{\mathbb S,\mathrm{op}}\norm{X}_F^2
		+
		\norm{B}_F
		\right)\norm{X}_F.
	\end{aligned}
	\]
	If \(X\in\mathcal L_0\), then \(\norm{X}_F\le R\), and hence
	\[
	\norm{\nabla f(X)}_F
	\le
	2\left(
	\norm{\mathcal A}_{\mathbb S,\mathrm{op}}R^2
	+
	\norm{B}_F
	\right)R
	=
	G_{\text{max}}.
	\]
	
	Finally, by Lemma~\ref{lem:projected_descent} and the selection rule in Algorithm~\ref{alg:safeadam_projected},
	\[
	f(X^{t+1})
	\le
	f(X_t^{\rm PG})
	\le
	f(X^t).
	\]
	Therefore, by induction,
	\[
	f(X^t)\le f(X^0),
	\]
	so \(X^t\in\mathcal L_0\) for all \(t\). Consequently,
	\[
	\norm{\nabla f(X^t)}_F\le G_{\text{max}}.
	\]
	This completes the proof.
\end{proof}

\begin{corollary}[Convergence in the gradient-mapping norm]
	\label{cor:projected_gradient_mapping}
	Under the assumptions of the convergence theorem for Algorithm~\ref{alg:safeadam_projected}, and with \(G_{\text{max}}\) defined in Lemma~\ref{lem:bounded_gradients}, it holds that
	\[
	\min_{0\le t<T}\norm{G_t}_F
	\le
	\sqrt{
		\frac{
			2\left(
			L_0+2^{2/3}L_{2/3} G_{\text{max}}^{2/3}
			\right)\Delta
		}{T}
	}.
	\]
\end{corollary}

\begin{proof}
	By the convergence theorem for the certified projected method,
	\[
	\min_{0\le t<T}
	\frac{\norm{G_t}_F^2}
	{L_0+2^{2/3}L_{2/3}\norm{\nabla f(X^t)}_F^{2/3}}
	\le
	\frac{2\Delta}{T}.
	\]
	Hence there exists \(t_*\in\{0,\dots,T-1\}\) such that
	\[
	\norm{G_{t_*}}_F^2
	\le
	\frac{2\Delta}{T}
	\left(
	L_0+2^{2/3}L_{2/3}\norm{\nabla f(X^{t_*})}_F^{2/3}
	\right).
	\]
	By Lemma~\ref{lem:bounded_gradients},
	\[
	\norm{\nabla f(X^{t_*})}_F\le G_{\text{max}}.
	\]
	Since the function
	\[
	a\mapsto L_0+2^{2/3}L_{2/3}a^{2/3}
	\]
	is nondecreasing, we obtain
	\[
	\norm{G_{t_*}}_F^2
	\le
	\frac{2\Delta}{T}
	\left(
	L_0+2^{2/3}L_{2/3} G_{\text{max}}^{2/3}
	\right).
	\]
	Taking the square root and using
	\[
	\min_{0\le t<T}\norm{G_t}_F
	\le
	\norm{G_{t_*}}_F
	\]
	gives the claimed bound.
\end{proof}

\section{Relation to the Gradient Projection Method}

\subsection{Gradient Projection Method}
For a function $f\colon \mathbb{R}^{n\times r}\to\mathbb{R}$ and a point $X\in\mathbb{R}^{n\times r}_{\ge 0}$ 
(with nonnegative entries) let us define the \emph{gradient projection} $\nabla_{\mathrm{proj}}f(X)$ 
componentwise as follows:
\[
(\nabla_{\mathrm{proj}}f(X))_{ij} =
\begin{cases}
	\displaystyle\frac{\partial f}{\partial X_{ij}}, & X_{ij}>0,\\[10pt]
	\displaystyle\min\!\left\{\frac{\partial f}{\partial X_{ij}},\;0\right\}, & X_{ij}=0 .
\end{cases}
\]
This is a standard generalization of the gradient for problems with nonnegativity constraints; 
a point $X^*\ge 0$ is stationary if and only if 
$\nabla_{\mathrm{proj}}f(X^*) = 0$.

This setting arises in problems of matrix completion~\cite{Candes09,Recht07,Jain13}, phase retrieval~\cite{Candes15}, symmetric nonnegative factorization~\cite{Dragomir21} and others.

At the same time the function $f \left(X\right) := F \left(X X^\top \right)$ loses convexity, as well as the property of global Lipschitz continuity of the gradient, but \textbf{acquires} the property of \textbf{generalized $(2/3, L_0, L_{2/3})$-smoothness}.

\begin{lemma}
	For any \(X\ge0\)
	\[
	\norm{\nabla_{\mathrm{proj}}f(X)}_F
	\le
	\norm{\nabla f(X)}_F.
	\]
\end{lemma}

\begin{proof}
	Let us consider the components of the gradient and of the projection.
	
	For the components with \(X_{ij}>0\) the values coincide.
	For the components of the gradient with \(X_{ij}=0\) the corresponding components of the gradient projection are equal either to \(0\), or to the negative part of the gradient, which does not increase the absolute value.
	Consequently, componentwise
	\[
	|(\nabla_{\mathrm{proj}}f(X))_{ij}|
	\le
	|\nabla_{ij}f(X)|.
	\]
	Summing the squares, we obtain the required statement.
\end{proof}

\begin{lemma}[Local equivalence]
	Let \(X\ge0\), \(\gamma>0\), and let for all positive components
	\[
	X_{ij}>0
	\quad\Rightarrow\quad
	X_{ij}-\gamma\nabla_{ij}f(X)>0
	\]
	hold.
	Then
	\[
	G_\gamma(X)=\nabla_{\mathrm{proj}}f(X).
	\]
\end{lemma}

\begin{proof}
	Let us consider the component \((i,j)\).
	
	\medskip
	\noindent\textbf{Case 1.} \(X_{ij}>0\).
	By the assumption the projection is not active, therefore
	\[
	(G_\gamma(X))_{ij}
	=
	\nabla_{ij}f(X).
	\]
	By definition the gradient projection is also
	\[
	(\nabla_{\mathrm{proj}}f(X))_{ij}
	=
	\nabla_{ij}f(X).
	\]
	
	\medskip
	\noindent\textbf{Case 2.} \(X_{ij}=0\).
	If \(\nabla_{ij}f(X)\le0\), then
	\[
	\mathrm{Proj}(0-\gamma\nabla_{ij}f(X))
	=
	-\gamma\nabla_{ij}f(X),
	\]
	and therefore
	\[
	(G_\gamma(X))_{ij}
	=
	\nabla_{ij}f(X).
	\]
	At the same time
	\[
	(\nabla_{\mathrm{proj}}f(X))_{ij}
	=
	\min\{\nabla_{ij}f(X),0\}
	=
	\nabla_{ij}f(X).
	\]
	
	If, however, \(\nabla_{ij}f(X)>0\), then
	\[
	\mathrm{Proj}(0-\gamma\nabla_{ij}f(X))=0,
	\]
	and therefore
	\[
	(G_\gamma(X))_{ij}=0.
	\]
	But also
	\[
	(\nabla_{\mathrm{proj}}f(X))_{ij}
	=
	\min\{\nabla_{ij}f(X),0\}=0.
	\]
	
	Thus, in all the cases the components coincide.
\end{proof}

\begin{corollary}
	If the iterates of the method lie in the region where the projection is inactive for the positive components, that is,
	\[
	X_{ij}^t > 0 \quad \Longrightarrow \quad X_{ij}^t \ge \delta > 0,
	\]
	and at the same time
	\[
	\gamma_t\norm{\nabla f(X^t)}_\infty \le \frac{\delta}{2},
	\]
	then the convergence with respect to the gradient mapping is equivalent to the convergence with respect to the gradient projection.
\end{corollary}

\begin{remark}[Why the projected guarantee is not split into two regimes]
	For the unconstrained methods, the convergence bound contains the same quantity in the numerator and in the denominator, namely \(\norm{\nabla f(X^t)}_F\). This allows one to obtain a self-contained inequality for
	\[
	g := \norm{\nabla f(X^t)}_F
	\]
	and to split it into the small-gradient and large-gradient regimes.
	
	For the projected method, the numerator contains the gradient mapping \(\norm{G_t}_F\), while the denominator contains the smoothness modulus evaluated at the ordinary gradient norm \(\norm{\nabla f(X^t)}_F\). These two quantities are not globally equivalent in the presence of active nonnegativity constraints. In particular, the gradient mapping may be zero at a constrained stationary point even though the ordinary gradient is nonzero.
	
	Therefore, the large-gradient cancellation leading to the \(T^{-3/4}\)-type rate is not available for \(\norm{G_t}_F\) without additional assumptions relating \(\norm{\nabla f(X^t)}_F\) and \(\norm{G_t}_F\). To obtain an unconditional bound on \(\norm{G_t}_F\), we upper-bound the denominator using the global gradient bound \(G_{\max}\), which yields a single square-root-type estimate.
\end{remark}

\section{Numerical Experiments}

\subsection{Tasks and Datasets}
\label{sec:tasks_and_datasets}

\paragraph{The task.}
All experiments in this section are performed for the symmetric nonnegative
matrix factorization problem (SymNMF)
\begin{equation}\label{eq:symnmf_exp}
	\min_{X \in \mathbb{R}^{n\times r}_{\ge 0}} \; f(X)
	\;=\; \frac12 \bigl\| M - X X^{\top} \bigr\|_F^2 ,
\end{equation}
which is a concrete instance of the factorized problem \eqref{main_func} with
\[
F(Y) = \frac12 \|M - Y\|_F^2,
\qquad
\mathcal{A} = \mathrm{Id},
\qquad
B = -M,
\qquad
L_F = \mu = 1,
\qquad
R \equiv 0 .
\]
Here $M \in \mathbb{R}^{n \times n}$ is a symmetric entrywise nonnegative
similarity matrix of $n$ data points and $r \ll n$ is the target rank. Since
the operator $\mathcal{A}$ is the identity, $F$ satisfies
Assumption~\ref{asm:quadratic} on the whole space with the exact constants
$\mu = L_F = 1$; consequently the constants $L_0$ and $L_{2/3}$ in the
generalized smoothness bound \eqref{final-p-L0Lp} are explicit, and this is
precisely the setting in which the Bregman-based Dyn-NoLips method
\cite{Dragomir21} is parameter-free. The similarity matrices $M$ are taken
from the public repository \cite{DragomirRepo} accompanying
\cite{Dragomir21}, in which the folder \texttt{SymNMF/data/} contains the
preprocessed similarity matrices for all four datasets used there.

\paragraph{The clustering meaning of \eqref{eq:symnmf_exp}.}
Problem \eqref{eq:symnmf_exp} is a clustering task rather than a pure matrix
approximation. If $\Theta \in \{0,1\}^{n \times r}$ denotes the binary
indicator matrix of a partition of the data points into $r$ disjoint groups,
then $\Theta\Theta^{\top}$ is the ideal block-diagonal affinity matrix, equal
to one exactly on the pairs of objects belonging to one cluster. Replacing the
combinatorial $\Theta$ by a continuous nonnegative factor $X$ yields
\eqref{eq:symnmf_exp}; the problem is therefore a continuous relaxation of the
clustering problem, and it is used as a probabilistic and graph clustering
technique \cite{Ding05,Kuang15}. Up to terms independent of $X$, the
SymNMF objective coincides with the kernel $k$-means criterion for a suitably
normalized affinity matrix \cite{Ding05}, and SymNMF achieves
state-of-the-art clustering accuracy on several text and image corpora
\cite{Kuang15}. After the factor $X$ is computed, the hard assignment is
recovered as
\begin{equation}\label{eq:cluster_assignment}
	\hat c_i = \arg\max_{1 \le j \le r} X_{ij},
\end{equation}
so that each row of $X$ is the soft membership vector of the $i$-th object and
$XX^{\top}$ is the denoised low-rank affinity. For the image data there is a
second, complementary interpretation: by the nonnegativity of $X$ the columns
of the factor are nonnegative basis images and each object is represented as
an additive combination of these parts without cancellation, which is the
classical parts-based representation of \cite{LeeSeung99}.

\paragraph{Construction of the similarity matrices.}
We follow exactly the preprocessing pipeline of \cite{Dragomir21,DragomirRepo},
which in turn reproduces Section~7.1 of \cite{Kuang15}. First, a similarity
graph over the data points is built: for text data the cosine similarity of
term-frequency vectors is used, and for image data a Gaussian kernel with the
self-tuning choice of the scale parameter. Second, the graph is sparsified by
retaining only the edges that connect the $k$-nearest neighbours of each point,
with
\[
k = \lfloor \log_2 n \rfloor + 1 .
\]
Third, $M$ is taken as a normalized version of the resulting adjacency matrix,
with entries $M_{ij} = \hat e_{ij}\, d_i^{-1/2} d_j^{-1/2}$, where
$d_i = \sum_{s} \hat e_{is}$ and $\hat e$ denotes the sparsified similarity
\cite{Kuang15}. The matrix $M$ is thus symmetric, nonnegative and sparse, and
is in general indefinite; its maximum entry does not exceed one and most of
its entries are close to zero. Sparsity is important for the cost of the
oracle: evaluating $f$ and $\nabla f$ costs
$O(nr^2 + \mathrm{nnz}(M)\,r)$, which makes the additional function calls of
the Armijo backtracking affordable (cf.\ the remark on computational cost).

\paragraph{The datasets of the repository.}
The repository \cite{DragomirRepo} provides two text and two image datasets.
\emph{TDT2} is the NIST Topic Detection and Tracking corpus of news stories of
the first half of 1998 broadcast and published by six sources (APW, NYT, VOA,
PRI, CNN, ABC); it contains 11{,}201 on-topic articles classified into 96
semantic categories, and after discarding multi-labelled documents and keeping
the largest categories one obtains a corpus of up to $n = 9394$ documents
\cite{DragomirThesis}. \emph{Reuters-21578} is the standard Reuters newswire
collection of 1987, restricted in \cite{DragomirThesis} to the largest 25
categories, which leaves $n = 7963$ documents; unlike TDT2, its documents may
carry several topic labels and its classes are heavily unbalanced.
\emph{CBCL} is Face Database~\#1 of the MIT Center for Biological and
Computational Learning, consisting of $n = 2429$ grayscale frontal-view face
images of size $19 \times 19$, hand-aligned on a common grid
\cite{Riesenhuber00}; this is the dataset on which the parts-based
interpretation of nonnegative factorizations was originally demonstrated
\cite{LeeSeung99}. \emph{COIL-20} contains $n = 1440$ images of size
$128 \times 128$ representing 20 objects observed under various angles
\cite{DragomirThesis}. In \cite{Dragomir21} the factorization is performed for
$r \in \{10, 20, 30, 40\}$, with $r$ playing the role of the number of
clusters.

\paragraph{The two problem classes used in this paper.}
We use three of the four datasets, which induce two structurally different
classes of SymNMF instances. The first class is formed by the text corpora
\textbf{TDT2} and \textbf{Reuters}: the affinity matrix is sparse, the
documents of one topic share a common vocabulary, so $M$ has a pronounced
block structure and the landscape of \eqref{eq:symnmf_exp} is comparatively
benign. The solved task is document clustering, i.e.\ grouping news articles
by topic with $r$ groups recovered through \eqref{eq:cluster_assignment}. The
second class is formed by the image collection \textbf{CBCL Faces}: all faces
share the same global structure (illumination, pose, background) and differ
mainly in fine details, so the affinity matrix is dense and slowly decaying in
spectrum, and the landscape of \eqref{eq:symnmf_exp} becomes ravine-like, with
the admissible step size differing by orders of magnitude between coordinate
directions. The task here is simultaneously clustering of face images into $r$
groups of similar appearance and learning of nonnegative face prototypes. This
dichotomy is the reason why a single heuristic does not work for both classes:
step enlargement with Armijo backtracking is effective on the predictable text
landscape, whereas coordinatewise preconditioning is needed on the
ill-conditioned image landscape (cf.\ Section~\ref{limitation_Dyn_NoLips}).

\paragraph{Experimental protocol.}
For all datasets and all methods we use the same rank $r = 20$, the same
tolerance $\varepsilon = 10^{-3}$ in the stopping criteria described below,
and the same initialization scheme: the starting point $X^0 \ge 0$ is drawn
randomly with small nonnegative entries. Following \cite{Dragomir21}, we
report the number of iterations, the wall-clock time and the final value of
the objective; where a reference optimum is available, the results may also be
presented through the normalized objective gap
$\bigl(f(X^t) - f_{\min}\bigr)/\|M\|_F^2$ averaged over several random
initializations. The baseline Dyn-NoLips and the projected gradient method
with Armijo line search are taken from \cite{DragomirRepo}
(\texttt{SymNMF/BPG.jl} and \texttt{SymNMF/PG.jl}), so that the comparison is
performed against the same implementation and on the same data files. Since we
study optimization methods, the reported quantities characterize the quality
of the solution of \eqref{eq:symnmf_exp} rather than external clustering
scores; a smaller objective value means a better fit of the low-rank cluster
structure of $M$, which makes it the natural optimization-side proxy for the
quality of the partition recovered by \eqref{eq:cluster_assignment}.

\subsection{Experiment 1: Text Data}
On text data the Dyn-NoLips method demonstrates a high convergence rate, which is due to the exact knowledge of the global smoothness constants for the quadratic loss function. 
However, as was noted in Section~\ref{limitation_Dyn_NoLips}, such a transfer is impossible for problems with unknown kernel parameters.

Among the methods that do not require manual tuning, the basic adaptive GD (Algorithm~\ref{alg:adaptive_gd_tyurin}) required 1014 iterations. 
The application of the proposed adaptive backtracking allowed us to reduce the number of iterations to 206 (an acceleration of $\times 4.9$), which confirms the efficiency of increasing the step with the Armijo condition on problems with a predictable landscape.

On the problem of symmetric nonnegative factorization (SymNMF)
with the matrix $M \in \mathbb{R}^{1000 \times 1000}$
and rank $r = 20$ the following results were obtained.

The adaptation of the factor $\alpha$ gives an additional acceleration
of $\times 2{.}5$ compared with a fixed $\alpha = 2$,
which confirms the efficiency of the proposed update rule.

\paragraph{Stopping criterion.}
For all proposed methods (Adaptive GD, Adaptive GD with backtracking, and SafeAdam-$\ell$), the algorithm is terminated when the relative norm of the gradient mapping falls below a prescribed tolerance $\varepsilon = 10^{-3}$:
\begin{equation}
	\frac{\|G_t\|_F}{\|G_0\|_F} < \varepsilon,
\end{equation}
where $G_t = \frac{1}{\gamma_t} \bigl( X_t - \mathrm{Proj}_{\ge 0}(X_t - \gamma_t \nabla f(X_t)) \bigr)$ is the gradient mapping at iteration $t$, $\gamma_t$ is the step size used at this iteration, and $G_0$ is the gradient mapping computed at the initial point $X_0$ using the initial step size $\gamma_0$. 

For the baseline Dyn-NoLips method, which operates with a different step-size rule, we use the relative norm of the projected gradient as the stopping criterion:
\begin{equation}
	\frac{\|\nabla_{\mathrm{proj}} f(X_t)\|_F}{\|\nabla_{\mathrm{proj}} f(X_0)\|_F} < \varepsilon.
\end{equation}
The local equivalence between $G_t$ and $\nabla_{\mathrm{proj}} f(X_t)$ proved in the theoretical part shows that both quantities are reasonable stationarity measures. However, this equivalence requires additional inactivity conditions for the nonnegativity constraints, which were not explicitly checked in the experiments. Therefore, the two stopping criteria should be regarded as method-specific and practically comparable, but not mathematically identical. In all experiments, the initial point $X_0$ was generated randomly with small nonnegative entries. To compare the final quality of the solutions, we additionally report the objective values $f(X)$ in the tables.
\subsection{TDT2}
\begin{figure}[H] 
	\centering
	\includegraphics[width=1.0\textwidth]{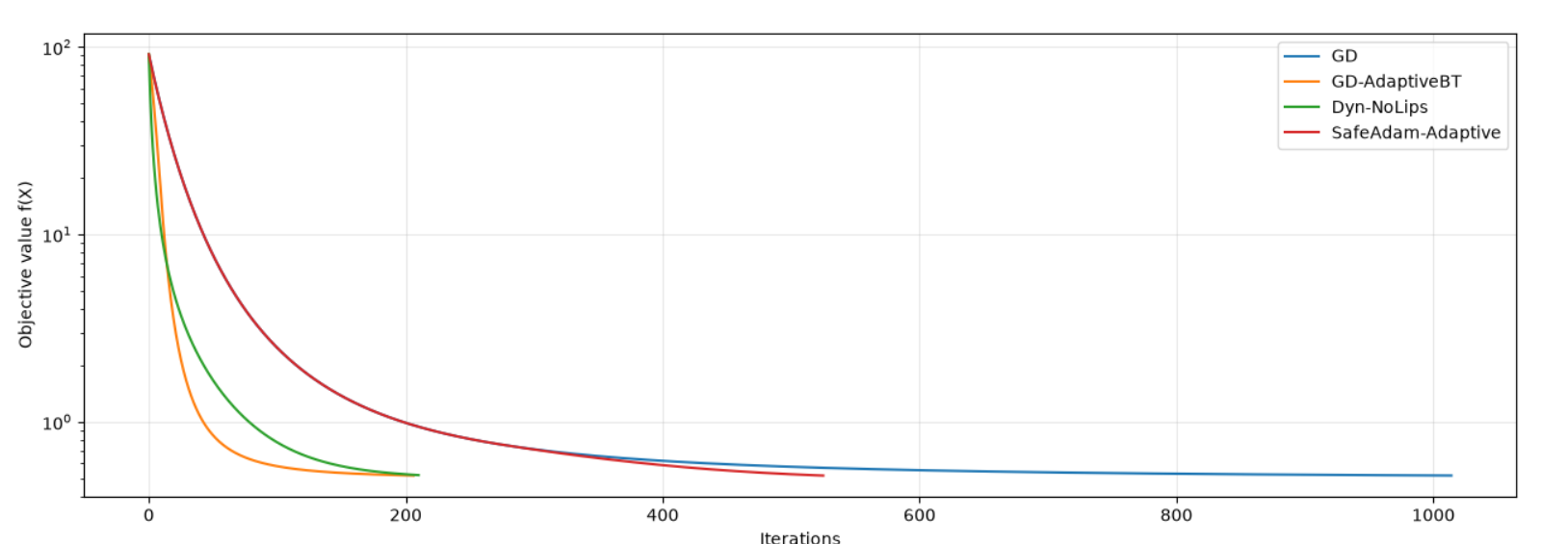} 
	\caption{Convergence rate of the methods on the TDT2 dataset}
	\label{fig:TDT2_picture_convergence} 
\end{figure}

\begin{table}[H]
	\centering
	\caption{Results on the TDT2 dataset ($r=20$, $\text{tol}=10^{-3}$)}
	\label{tab:tdt2_results}
	\begin{tabular}{lccc} 
		\toprule
		Method & Iterations & Time (s) & Final $f(X)$ \\
		\midrule
		GD & 1014 & 19.16 & 0.516649 \\
		GD-Adaptive & 206 & 6.64 & 0.516670 \\
		Dyn-NoLips & 210 & 6.53 & 0.519804 \\
		SafeAdam-Adaptive & 525 & 20.61 & 0.516563 \\
		\bottomrule
	\end{tabular}
\end{table}

In order to evaluate the proposed methods we carried out experiments on two fundamentally different problems of symmetric nonnegative factorization (SymNMF):
\begin{enumerate}
	\item \textbf{Text data (Reuters/TDT2):} a sparse similarity matrix, a specific low-rank structure of the gradient.
	\item \textbf{Face images (CBCL):} a dense similarity matrix, a ravine-like landscape of the loss function.
\end{enumerate}

\subsection{Reuters Dataset}
\begin{figure}[H] 
	\centering
	\includegraphics[width=1.0\textwidth]{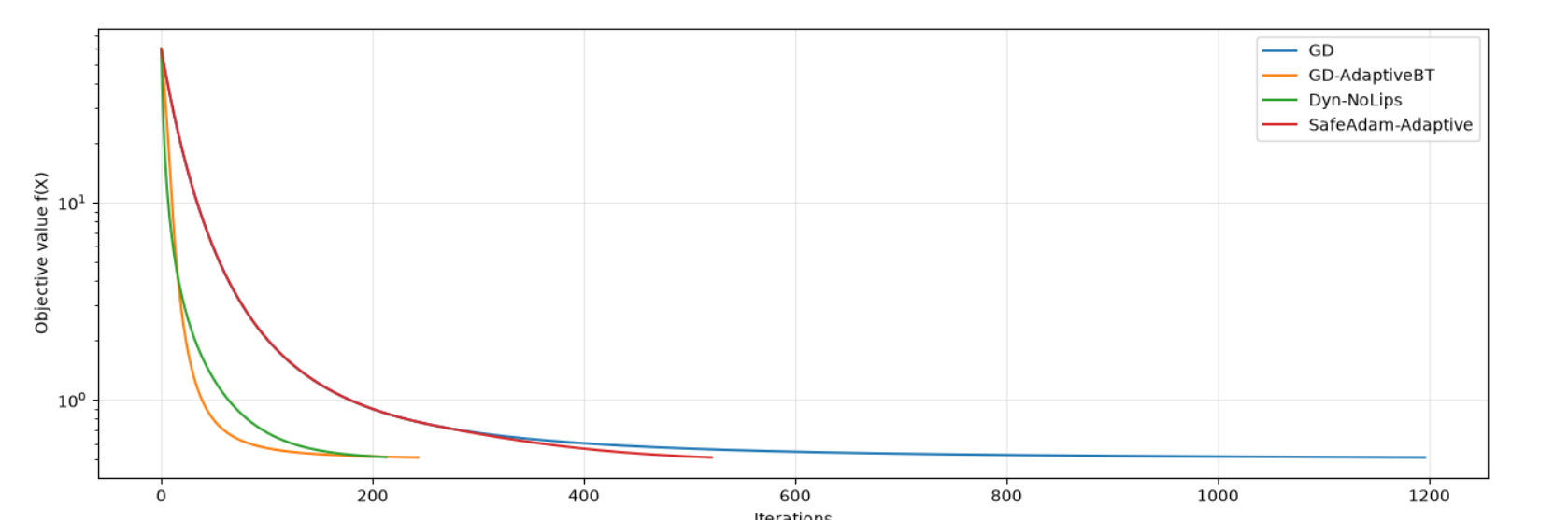} 
	\caption{Convergence rate of the methods on the Reuters dataset}
	\label{fig:Reuters_Picture_convergence} 
\end{figure}

\begin{table}[H]
	\centering
	\caption{Results on the Reuters dataset ($r=20$, $\text{tol}=10^{-3}$)}
	\label{tab:reuters_results}
	\begin{tabular}{lccc} 
		\toprule
		Method & Iterations & Time (s) & Final $f(X)$\\
		\midrule
		GD & 1196 & 20.19 & 0.509812 \\
		GD-Adaptive & 243 & 4.38 & 0.509656 \\
		Dyn-NoLips & 213 & 5.56 & 0.511219 \\
		SafeAdam-Adaptive & 521 & 17.92 & 0.509717 \\
		\bottomrule
	\end{tabular}
\end{table}

\subsection{Experiment 2: Face Images (CBCL)}
On the CBCL dataset (2446 face images, $32 \times 32$ pixels, rank $r=20$) the situation changes fundamentally (Table~\ref{tab:cbcl_results}).

\begin{table}[H]
	\centering
	\caption{Results on the CBCL Faces dataset ($r=20$, $\text{tol}=10^{-3}$)}
	\label{tab:cbcl_results}
	\begin{tabular}{lccc}
		\toprule
		\textbf{Method} & \textbf{Iterations} & \textbf{Time (s)} & \textbf{Final $f(X)$} \\
		\midrule
		GD (Algorithm~\ref{alg:adaptive_gd_tyurin}) & 15000 & 5.80 & 0.4906 \\
		GD-Adaptive & 15000 & 41.97 & 0.4261 \\
		Dyn-NoLips & 9115 & 6.04 & 0.4188 \\
		\textbf{SafeAdam-$\ell$ (Adapt.)} & \textbf{6764} & \textbf{5.46} & \textbf{0.4186} \\
		\bottomrule
	\end{tabular}
\end{table}

\textbf{Analysis of the results:}
\begin{itemize}
	\item \textbf{Degradation of Dyn-NoLips:} the method required 9115 iterations, which confirms our thesis on the limitations of Dyn-NoLips beyond simple quadratic problems.
	\item \textbf{Failure of backtracking:} GD-Adaptive did not converge within the iteration limit, spending 7 times more time per iteration. This is related to the ravine-like structure of the landscape on images: the increased step violated the Armijo condition, which led to multiple calls of the function oracle and to a forced rollback to the conservative step.
	\item \textbf{Victory of SafeAdam-$\ell$:} in contrast to the text data, on images the gradient has a denser spectral structure. The coordinatewise adaptation in Adam acts as an efficient preconditioner that equalizes the scales of the features. In combination with the adaptive threshold $c_{\mathrm{angle}}$ (warm-up) this ensured the best convergence rate (6764 iterations) and the minimal wall-clock time.
\end{itemize}

\subsection{General Conclusion}
The experiments demonstrate that the proposed theoretical approach of $(2/3, L_0, L_{2/3})$-smoothness and of the $q$-function provides a reliable foundation for the construction of algorithms. However, the choice of a particular heuristic (backtracking for the step length or an adaptive direction of the Adam type) has to be determined by the structure of the problem:
\begin{itemize}
	\item For problems with a \textit{predictable landscape} and a cheap oracle the \textbf{adaptive backtracking} is optimal.
	\item For problems with a \textit{ravine-like landscape} and a complicated spectral structure of the gradient \textbf{SafeAdam-$\ell$} is optimal.
\end{itemize}
In both cases our approach outperforms or is comparable with Dyn-NoLips, while not requiring the tuning of the kernel parameters of relative smoothness.

\section{Conclusion}
In this work, modified methods of adaptive gradient descent and Adam are proposed for problems with generalized $(2/3,L_0,L_{2/3})$-smoothness. The methods use the Adam direction, a descent angle check, a fallback to the normalized gradient, and a step length consistent with the $q$-function of generalized smoothness. Global convergence with respect to the gradient norm is proved, and explicit convergence rates are obtained. The results are applicable to problems of low-rank matrix factorization in the Burer--Monteiro form, in particular to the SymNMF problem.

The obtained guarantees feature two explicitly regimes, which naturally arise from the structure of the generalized smoothness bound:
\begin{itemize}
	\item \textbf{Small-gradient regime:} When the gradient norm is small, the constant part $L_0$ of the generalized smoothness dominates. In this regime, the convergence rate is $O(T^{-1/2})$, which corresponds to the standard complexity guarantee $T = O(\varepsilon^{-2})$ to achieve an $\varepsilon$-stationary point.
	
	\item \textbf{Large-gradient regime:} When the gradient norm is large, the fractional-power part $L_{2/3}\|\nabla f\|^{2/3}$ dominates. In this regime, the convergence rate accelerates to $O(T^{-3/4})$, which corresponds to the improved complexity guarantee $T = O(\varepsilon^{-4/3})$.
\end{itemize}

For the unconstrained methods, the proof is based on a two-stage telescoping
split over small-gradient and large-gradient iterations. This yields additive
no-max complexity bounds and avoids relying on monotonicity of the gradient norm.

This dichotomy explains both the theoretical bounds and the observed practical behaviour of the proposed methods, demonstrating their robustness and adaptability across different stages of the optimization process. Furthermore, unlike Bregman-based approaches that require global a priori knowledge to fix the geometry of the space, our Euclidean framework allows for on-the-fly adaptation of the smoothness constants, making it highly efficient for ill-conditioned or ravine-like landscapes.

Finally, we would like to disclose the specific contributions and the tools utilized in the preparation of this manuscript. The core conceptual idea of establishing the $(2/3, L_0, L_{2/3})$-generalized smoothness for the Burer--Monteiro factorized objective was originally proposed by the authors. The technical derivations and formal mathematical proofs were constructed with the assistance of the artificial intelligence models Qwen 3.8 Max and Deepseek v4 pro. All computational experiments, including the design of the numerical protocols and the analysis of the results, were conducted manually by the authors.

	\appendix
	
	\section*{Appendix}
	
	\begin{algorithm}[H]
		\caption{Unconstrained SafeAdam-\(\ell\)}
		\label{alg:safeadam_unconstrained}
		\Input{\(X^0\), \(\beta_1,\beta_2\in[0,1)\), \(\varepsilon_{\rm Adam}>0\), \(c_{\rm angle}\in(0,1]\), \(L_0,L_{2/3}\), \(T\)}
		\Output{\(X^T\)}
		$m_{-1}\gets0$, $v_{-1}\gets0$\;
		\For{$t=0,1,\dots,T-1$}{
			$g_t\gets\nabla f(X^t)$\;
			$a_t\gets\norm{g_t}_F$\;
			\If{$a_t=0$}{
				\Return $X^t$\;
			}
			$\gamma_t^{\rm PG}\gets q(a_t;a_t)/a_t$\;
			$X_t^{\rm PG}\gets X^t-\gamma_t^{\rm PG}g_t$\;
			$m_t\gets\beta_1m_{t-1}+(1-\beta_1)g_t$\;
			$v_t\gets\beta_2v_{t-1}+(1-\beta_2)g_t\circ g_t$\;
			$\hat m_t\gets m_t/(1-\beta_1^{t+1})$\;
			$\hat v_t\gets v_t/(1-\beta_2^{t+1})$\;
			$u_t\gets \hat m_t/(\sqrt{\hat v_t}+\varepsilon_{\rm Adam})$\;
			\If{$u_t\neq0$}{
				$d_t^{\rm Adam}\gets u_t/\norm{u_t}_F$\;
			}
			\Else{
				$d_t^{\rm Adam}\gets0$\;
			}
			\If{$\inprod{g_t}{d_t^{\rm Adam}}\ge c_{\rm angle}a_t$}{
				$b_t\gets\inprod{g_t}{d_t^{\rm Adam}}$\;
				$\eta_t^{\rm Adam}\gets q(b_t;a_t)$\;
				$X_t^{\rm Adam}\gets X^t-\eta_t^{\rm Adam}d_t^{\rm Adam}$\;
			}
			\Else{
				$X_t^{\rm Adam}\gets X_t^{\rm PG}$\;
			}
			\If{$f(X_t^{\rm Adam})\le f(X_t^{\rm PG})$}{
				$X^{t+1}\gets X_t^{\rm Adam}$\;
			}
			\Else{
				$X^{t+1}\gets X_t^{\rm PG}$\;
			}
		}
	\end{algorithm}

	\begin{algorithm}[H]
		\caption{Certified projected SafeAdam-\(\ell\)}
		\label{alg:safeadam_projected}
		\Input{\(X^0\ge0\), \(\beta_1,\beta_2\in[0,1)\), 
			\(\varepsilon_{\rm Adam}>0\), \(c_{\rm angle}\in(0,1]\), 
			\(L_0,L_{2/3}\), \(T\)}
		\Output{\(X^T\)}
		$m_{-1}\gets0$, $v_{-1}\gets0$\;
		\For{$t=0,1,\dots,T-1$}{
			$g_t\gets\nabla f(X^t)$\;
			$a_t\gets\norm{g_t}_F$\;
			\If{$a_t=0$}{
				\Return $X^t$\;
			}
			$\gamma_t^{\rm PG}\gets q(a_t;\,a_t)/a_t$\;
			$X_t^{\rm PG}\gets\Proj\bigl(X^t-\gamma_t^{\rm PG}g_t\bigr)$\;
			$G_t\gets \bigl(X^t-X_t^{\rm PG}\bigr)/\gamma_t^{\rm PG}$\;
			$m_t\gets\beta_1m_{t-1}+(1-\beta_1)g_t$\;
			$v_t\gets\beta_2v_{t-1}+(1-\beta_2)\,g_t\circ g_t$\;
			$\hat m_t\gets m_t/(1-\beta_1^{t+1})$\;
			$\hat v_t\gets v_t/(1-\beta_2^{t+1})$\;
			$u_t\gets \hat m_t/\bigl(\sqrt{\hat v_t}+\varepsilon_{\rm Adam}\bigr)$\;
			\If{$u_t\neq0$}{
				$d_t^{\rm Adam}\gets u_t/\norm{u_t}_F$\;
			}
			\Else{
				$d_t^{\rm Adam}\gets0$\;
			}
			\If{$\inprod{g_t}{d_t^{\rm Adam}}\ge c_{\rm angle}\,a_t$}{
				$b_t\gets\inprod{g_t}{d_t^{\rm Adam}}$\;
				$\eta_t^{\rm Adam}\gets q(b_t;\,a_t)$\;
				$X_t^{\rm Adam}\gets\Proj\bigl(X^t-\eta_t^{\rm Adam}d_t^{\rm Adam}\bigr)$\;
			}
			\Else{
				$X_t^{\rm Adam}\gets X_t^{\rm PG}$\;
			}
			\If{$f(X_t^{\rm Adam})\le f(X_t^{\rm PG})$}{
				$X^{t+1}\gets X_t^{\rm Adam}$\;
			}
			\Else{
				$X^{t+1}\gets X_t^{\rm PG}$\;
			}
		}
	\end{algorithm}
	
	\section{Problems Requiring Grid Search for Kernel Parameters}
	\label{sec:grid_search}
	
	In the previous sections we derived sufficient conditions on the kernel parameters that guarantee relative smoothness.
	For the norm kernel
	\[
	h_N(X) = \frac{\alpha}{4}\|X\|^4 + \frac{\sigma}{2}\|X\|^2
	\]
	the conditions are
	\[
	\alpha \ge 6L_F, \qquad \sigma \ge 2\|\nabla F(0)\|.
	\]
	For the Gram kernel
	\[
	h_G(X) = \frac{\alpha}{4}\|X\|^4 + \frac{\beta}{4}\|X^\top X\|^2 + \frac{\sigma}{2}\|X\|^2
	\]
	the conditions are
	\[
	\alpha \ge 2(L_F - \mu_F), \qquad \beta \ge 2L_F, \qquad \sigma \ge 2\|\nabla F(0)\|.
	\]
	These conditions are valid and can be used directly when the global constants
	\(L_F\), \(\mu_F\), \(\|\nabla F(0)\|\), and \(\|\mathcal A\|_{\mathbb S,\mathrm{op}}\)
	are known exactly.
	However, in many practical low‑rank optimization problems these constants are either unknown, computationally expensive to estimate, or the sufficient conditions are too conservative, leading to prohibitively small step sizes.
	In such cases the kernel parameters must be tuned by grid search.
	Below we list the problem classes where this is unavoidable.
	
	\subsection{Matrix Sensing with Unknown Operator}
	\paragraph{Formulation.}
	\begin{equation}
		\min_{Y \in \mathbb S^{n\times n}} F(Y) = \frac{1}{2}\|\mathcal A(Y) - b\|_F^2,
	\end{equation}
	where \(\mathcal A: \mathbb S^{n\times n} \to \mathbb R^m\) is a linear operator and \(b \in \mathbb R^m\).
	The Burer--Monteiro factorization writes \(Y = XX^\top\) with \(X \in \mathbb R^{n\times r}\).
	
	\paragraph{Why grid search is mandatory.}
	The smoothness and strong convexity constants are
	\(L_F = \|\mathcal A\|_{\mathrm{op}}^2\) and \(\mu_F = \sigma_{\min}^2(\mathcal A)\).
	If \(\mathcal A\) is given only implicitly (e.g., as a black‑box linear map) or its spectrum is not computed, these constants are unavailable.
	Even if computed, \(\sigma_{\min}(\mathcal A)\) may be zero or extremely small, making \(\mu_F \approx 0\) and rendering the sufficient conditions inapplicable.
	Hence, one must resort to grid search over \(\alpha, \sigma\) (or \(\alpha, \beta, \sigma\) for the Gram kernel).
	
	\subsection{Matrix Completion with Restricted Strong Convexity}
	\paragraph{Formulation.}
	\begin{equation}
		\min_{Y \in \mathbb S^{n\times n}} F(Y) = \frac{1}{2}\|\mathcal P_\Omega(Y - M)\|_F^2,
	\end{equation}
	where \(\Omega \subset [n]\times[n]\) is the set of observed entries and \(\mathcal P_\Omega\) is the orthogonal projection onto \(\Omega\).
	The factorized problem is \(\min_X f(X) = F(XX^\top)\).
	
	\paragraph{Why grid search is mandatory.}
	The function \(F\) is convex but \emph{not} strongly convex on the whole space.
	The strong convexity parameter \(\mu > 0\) exists only on the subspace of low‑rank matrices, and the restricted constant \(\mu_{\text{rest}}\) depends on \(\Omega\) and the rank \(r\).
	Since the generalized smoothness constants \(L_0\) and \(L_{2/3}\) in Theorem~\ref{thm:lower-bound} are derived using a global \(\mu > 0\), they cannot be applied directly.
	Consequently, the kernel parameters must be tuned by grid search.
	
	\subsection{Phase Retrieval}
	\paragraph{Formulation.}
	\begin{equation}
		\min_{Y \in \mathbb S^{n\times n}} F(Y) = \frac{1}{2}\sum_{i=1}^m \left(\langle a_i a_i^\top, Y\rangle - b_i\right)^2,
	\end{equation}
	with \(Y = XX^\top\), \(X \in \mathbb R^{n\times r}\).
	
	\paragraph{Why grid search is mandatory.}
	The linear operator is \(\mathcal A = \sum_{i=1}^m (a_i a_i^\top)(a_i a_i^\top)^\top\).
	Its condition number can be very large, and the constants \(L_F = \|\mathcal A\|_{\mathrm{op}}\) and \(\mu_F = \sigma_{\min}(\mathcal A)\) are either unknown or produce extremely conservative sufficient conditions.
	Grid search over the kernel parameters is therefore necessary.
	
	\subsection{Euclidean Distance Matrix Completion with Irregular Graphs}
	\paragraph{Formulation.}
	\begin{equation}
		\min_{X \in \mathbb R^{n\times r}} f(X) = \frac{1}{2}\sum_{(i,j)\in \Omega} \left(\|X_i - X_j\|^2 - d_{ij}\right)^2,
	\end{equation}
	where \(\Omega \subset [n]\times[n]\) is the set of known pairwise distances.
	
	\paragraph{Why grid search is mandatory.}
	The Lipschitz constant for this problem is
	\[
	L_{\text{EDM}} = 9 \max_{i=1,\dots,n} |\{j : (i,j) \in \Omega\}|.
	\]
	For irregular observation graphs, this maximum degree can be very large, making the sufficient condition \(\alpha \ge 6L_{\text{EDM}}\) extremely conservative and the resulting step size tiny.
	In practice, one must tune \(\alpha\) (and \(\sigma\), and for the Gram kernel also \(\beta\)) by grid search to obtain reasonable performance.
	
	\subsection{Ill‑Conditioned Matrix Factorization}
	\paragraph{Formulation.}
	\begin{equation}
		\min_{Y \in \mathbb S^{n\times n}} F(Y) = \frac{1}{2}\langle \mathcal A Y, Y\rangle + \langle B, Y\rangle,
	\end{equation}
	with \(Y = XX^\top\), where \(\mathcal A\) is positive definite on symmetric matrices with condition number \(\kappa_F = L_F / \mu_F \gg 1\).
	
	\paragraph{Why grid search is mandatory.}
	The sufficient conditions for the Gram kernel depend on \(L_F\) and \(\mu_F\).
	When \(\kappa_F\) is large, the gap between the analytical parameters and the optimal ones is significant.
	The quality of the Gram kernel approximation is governed by \(1 - \mu_F/L_F\), so for ill‑conditioned problems the kernel parameters must be tuned by grid search.
	
	\subsection{Large Rank Factorizations}
	\paragraph{Formulation.}
	Consider the general quadratic
	\[
	F(Y) = \frac{1}{2}\langle \mathcal A Y, Y\rangle + \langle B, Y\rangle
	\]
	with \(Y = XX^\top\) and \(X \in \mathbb R^{n\times r}\), where the rank \(r\) is large (e.g., \(r \ge 20\)).
	
	\paragraph{Why grid search is mandatory.}
	The derived constants in Theorem~\ref{thm:lower-bound} are
	\begin{equation}
		L_{2/3} = 6\|\mathcal A\|_{\mathbb S,\mathrm{op}}\left(\frac{r}{\mu}\right)^{2/3}, \qquad
		L_0 = 6\|\mathcal A\|_{\mathbb S,\mathrm{op}}\frac{2r\|B\|_F}{\mu} + 2\|B\|_F.
	\end{equation}
	These bounds grow with \(r\).
	For large \(r\), they become extremely conservative, leading to very small step sizes.
	Grid search over the kernel parameters (or over \(L_0, L_{2/3}\)) is necessary to achieve practical convergence.
	
	\subsection{Nonzero Linear Term with Unknown Norm}
	\paragraph{Formulation.}
	\begin{equation}
		F(Y) = \frac{1}{2}\langle \mathcal A Y, Y\rangle + \langle B, Y\rangle,
	\end{equation}
	with \(B \neq 0\) and either \(\|B\|_F\) unknown or very large.
	
	\paragraph{Why grid search is mandatory.}
	The norm \(\|B\|_F\) enters both \(L_0\) and the kernel parameter \(\sigma \ge 2\|\nabla F(0)\| = 2\|B\|_F\).
	If \(B\) is not explicitly given or its norm is large, the analytical \(\sigma\) produces a tiny step size.
	In practice, one must tune \(\sigma\) (and possibly \(\alpha\)) by grid search.
	
	\subsection{Symmetric Nonnegative Matrix Factorization with Arbitrary \(M\)}
	\paragraph{Formulation.}
	\begin{equation}
		\min_{X \ge 0} f(X) = \frac{1}{2}\|M - XX^\top\|_F^2,
	\end{equation}
	where \(M \in \mathbb S^{n\times n}\) is a symmetric nonnegative matrix and \(X \in \mathbb R^{n\times r}_{\ge 0}\).
	
	\paragraph{Why grid search is mandatory.}
	For this specific problem, \(L_F = 1\), \(\mu_F = 1\), and \(\|\nabla F(0)\| = \|M\|_F\).
	Thus the analytical parameters are known: \(\alpha = 6\), \(\sigma = 2\|M\|_F\).
	However, if \(M\) is not normalized and has a large scale, \(\sigma\) becomes huge, resulting in slow convergence.
	Moreover, the active nonnegativity constraints alter the effective geometry, making the Euclidean‑based sufficient conditions loose.
	In practice, one reduces \(\sigma\) empirically via grid search to accelerate convergence, as reported in our numerical experiments.
	
	\subsection{Compressed Sensing with Restricted Strong Convexity}
	This is a special case of Section~\ref{sec:grid_search} (Matrix Completion) and is mentioned here for completeness.
	The same reasoning applies: without a global strong convexity parameter, the analytical constants are unavailable, and grid search is required.
	
	\paragraph{Remark.}
	In all the above cases, the necessity of grid search stems from either (i) missing or expensive global constants, or (ii) overly conservative sufficient conditions due to problem structure (large \(r\), ill‑conditioning, restricted strong convexity, active constraints).
	When the constants are known and the problem is well‑conditioned with small \(r\), the analytical parameter choices from Propositions~1 and~2 are sufficient and no tuning is required.
\end{document}